\documentclass[10pt]{amsart}

\usepackage[margin=1in,letterpaper,portrait]{geometry}
\usepackage{amsmath, amssymb, amsthm, amscd, verbatim}
\usepackage{tikz}
\usepackage{enumerate}
\usetikzlibrary{matrix,arrows}

\newcommand{\mb}{\mathbb}

\newcommand{\im}{{\operatorname{im}}}
\newcommand{\del}{\partial}
\newcommand{\coker}{{\rm coker}}

\newcommand{\ZZ}{\mathbb{Z}}
\newcommand{\QQ}{\mathbb{Q}}
\newcommand{\RR}{\mathbb{R}}
\newcommand{\CC}{\mathbb{C}}

\newcommand{\Hom}{\mathrm{Hom}}
\newcommand{\one}{\mathbf{1}}

\newcommand{\Double}{\mathrm{Double}}
\newcommand{\Syl}{\mathrm{Syl}}
\newcommand{\redZH}{\overline{\ZZ H}}
\newcommand{\delT}{\delta}
\newcommand{\sym}{\mathrm{sym}}
\newcommand{\antisym}{\mathrm{skew}}
\newcommand{\crown}{\mathrm{Crown}}

\begin{document}

\theoremstyle{plain}
  \newtheorem{theorem}{Theorem}[section]
  \newtheorem{proposition}[theorem]{Proposition}
  \newtheorem{lemma}[theorem]{Lemma}
  \newtheorem{corollary}[theorem]{Corollary}
  \newtheorem{conjecture}[theorem]{Conjecture}
\theoremstyle{definition}
  \newtheorem{definition}[theorem]{Definition}
  \newtheorem{example}[theorem]{Example}
  \newtheorem{examples}[theorem]{Examples}
  \newtheorem{question}[theorem]{Question}
  \newtheorem{problem}[theorem]{Problem}
 \theoremstyle{remark}
  \newtheorem{remark}[theorem]{Remark}

\setcounter{MaxMatrixCols}{20}

\numberwithin{equation}{section}

\tikzstyle{vertex}=[circle,fill=black!20,minimum size=15pt,inner sep=0pt]
\tikzstyle{selected vertex} = [vertex, fill=red!24]
\tikzstyle{edge} = [draw,thin,-]
\tikzstyle{weight} = [font=\small]
\tikzstyle{selected edge} = [draw,line width=5pt,-,red!50]
\tikzstyle{ignored edge} = [draw,line width=5pt,-,black!20]


\thanks{This project was undertaken during an REU at the Univ. of Minnesota School of Mathematics in Summer 2012,
co-mentored by G. Musiker, P. Pylyavskyy, V. Reiner and D. Stanton, and supported by NSF RTG grant number DMS-1148634.
The authors thank the participants during this REU for their helpful comments and suggestions throughout.}

\title[Critical groups of covering, voltage and signed graphs]%
{Critical groups of covering, voltage and signed graphs} 

\author{Victor Reiner}
\address{School of Mathematics, Univ. of Minnesota\\ Minneapolis, MN 55455}
\author{Dennis Tseng} 
\address{Dept. of Mathematics, Massachusetts Institute of Technology \\ Cambridge, MA 02139}

\email{reiner@math.umn.edu}
\email{dennisctseng@gmail.com} 

\begin{abstract}
Graph coverings are known to induce 
surjections of their critical groups.  Here we describe 
the kernels of these morphisms in terms of data parametrizing the covering. 
Regular coverings are parametrized by voltage graphs, and the above kernel can be identified
with a naturally defined voltage graph critical group.  
For double covers, the voltage graph is a signed graph, and the theory takes a 
particularly pleasant form, leading also to a theory of double covers of signed graphs.  
\end{abstract}

\maketitle


\section{Introduction}
\label{Introduction}

This paper studies graph coverings and critical groups for undirected {\it multigraphs}
$G=(V,E)$; here $E$ is a multiset of edges, with self-loops allowed.
An example graph covering $\tilde{G} \rightarrow G$ is shown here,
where the map sends an edge or vertex of $\tilde{G}$ to the corresponding edge or vertex of $G$ by ignoring the $+/-$ subscript:

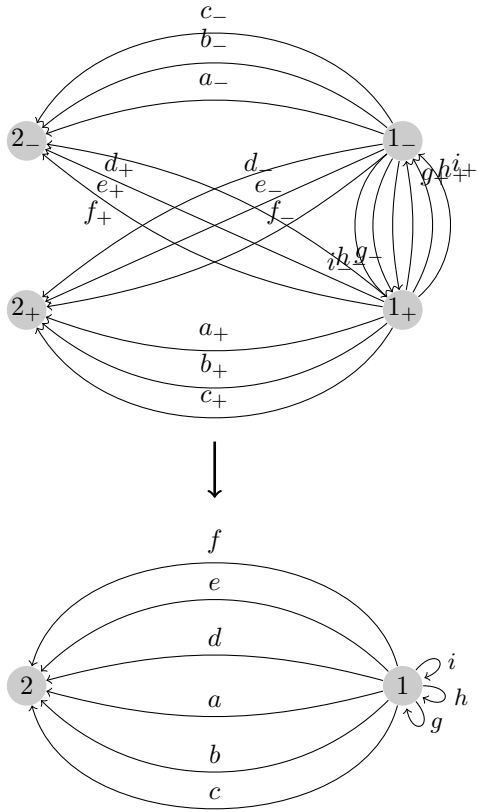
\begin{center}
\begin{tikzpicture}[->,scale=0.25, auto,swap]
    \foreach \pos/\name in {{(0,0)/2_+},{(20,0)/1_+},{(0,9)/2_-},{(20,9)/1_-},
                            {(0,-20)/2},{(20,-20)/1}}
        \node[vertex] (\name) at \pos {$\name$};

    \foreach \source/ \dest/\arclabel in {1_+/2_+/a_+}
        \path (\source) edge[bend left=20] node[auto] {$\arclabel$} (\dest);
    \foreach \source/ \dest/\arclabel in {1_+/2_+/b_+}
        \path (\source) edge[bend left=40] node[auto] {$\arclabel$} (\dest);
    \foreach \source/ \dest/\arclabel in {1_+/2_+/c_+}
        \path (\source) edge[bend left=60] node[auto] {$\arclabel$} (\dest);
    \foreach \source/ \dest/\arclabel in {1_+/2_-/d_+}
        \path (\source) edge[bend left=-15] node[left, near end] {$\arclabel$} (\dest);
    \foreach \source/ \dest/\arclabel in {1_+/2_-/e_+}
        \path (\source) edge[bend left=0] node[left, near end] {$\arclabel$} (\dest);
    \foreach \source/ \dest/\arclabel in {1_+/2_-/f_+}
        \path (\source) edge[bend left=15] node[left, near end] {$\arclabel$} (\dest);

    \foreach \source/ \dest/\arclabel in {1_-/2_-/a_-}
        \path (\source) edge[bend right=20] node[auto] {$\arclabel$} (\dest);
    \foreach \source/ \dest/\arclabel in {1_-/2_-/b_-}
        \path (\source) edge[bend right=40] node[auto] {$\arclabel$} (\dest);
    \foreach \source/ \dest/\arclabel in {1_-/2_-/c_-}
        \path (\source) edge[bend right=60] node[auto] {$\arclabel$} (\dest);
    \foreach \source/ \dest/\arclabel in {1_-/2_+/d_-}
        \path (\source) edge[bend left=-15] node[left, near start] {$\arclabel$} (\dest);
    \foreach \source/ \dest/\arclabel in {1_-/2_+/e_-}
        \path (\source) edge[bend left=0] node[left, near start] {$\arclabel$} (\dest);
    \foreach \source/ \dest/\arclabel in {1_-/2_+/f_-}
        \path (\source) edge[bend left=15] node[left, near start] {$\arclabel$} (\dest);

   \foreach \source/ \dest/\arclabel in {1_+/1_-/g_+}
        \path (\source) edge[bend right=10] node[near end] {$\arclabel$} (\dest);
    \foreach \source/ \dest/\arclabel in {1_+/1_-/h_+}
        \path (\source) edge[bend right=30] node[near end] {$\arclabel$} (\dest);
    \foreach \source/ \dest/\arclabel in {1_+/1_-/i_+}
        \path (\source) edge[bend right=50] node[near end] {$\arclabel$} (\dest);

    \foreach \source/ \dest/\arclabel in {1_-/1_+/g_-}
        \path (\source) edge[bend right=10] node[left, near end] {$\arclabel$} (\dest);
    \foreach \source/ \dest/\arclabel in {1_-/1_+/h_-}
        \path (\source) edge[bend right=30] node[left, near end] {$\arclabel$} (\dest);
    \foreach \source/ \dest/\arclabel in {1_-/1_+/i_-}
        \path (\source) edge[bend right=50] node[left, near end] {$\arclabel$} (\dest);

\draw[arrows=->,line width=1pt](10,-7)--(10,-10);

    \foreach \source/ \dest/\arclabel in {1/2/a}
        \path (\source) edge[bend left=15] node[auto] {$\arclabel$} (\dest);
    \foreach \source/ \dest/\arclabel in {1/2/b}
        \path (\source) edge[bend left=45] node[auto] {$\arclabel$} (\dest);
    \foreach \source/ \dest/\arclabel in {1/2/c}
        \path (\source) edge[bend left=75] node[auto] {$\arclabel$} (\dest);
    \foreach \source/ \dest/\arclabel in {1/2/d}
        \path (\source) edge[bend right=15] node[auto] {$\arclabel$} (\dest);
    \foreach \source/ \dest/\arclabel in {1/2/e}
        \path (\source) edge[bend right=45] node[auto] {$\arclabel$} (\dest);
    \foreach \source/ \dest/\arclabel in {1/2/f}
        \path (\source) edge[bend right=75] node[auto] {$\arclabel$} (\dest);

\foreach \source in {1}    
	\path[every node/.style={font=\sffamily\small}]
        (\source)   edge[in=20, out=50, loop] node[right] {$i$}  (\source);
\foreach \source in {1}    
	\path[every node/.style={font=\sffamily\small}]
        (\source)   edge[in=-30, out=0, loop] node[right] {$h$}  (\source);
\foreach \source in {1}    
	\path[every node/.style={font=\sffamily\small}]
        (\source)   edge[in=-80, out=-50, loop] node[right] {$g$}  (\source);
\end{tikzpicture}
\end{center}

The {\it critical group} $K(G)$
is a subtle isomorphism invariant of $G$ in the form of a finite abelian group,
whose cardinality is the number of {\it maximal forests} in $G$.  To present $K(G)$,
one can introduce the {\it (signed) node-edge incidence} 
matrix $\del:=\del_G$ for $G$ having rows indexed by $V$, columns indexed by $E$, as we now explain.  
One defines $\del$ by first fixing an arbitrary orientation of the edge 
set $E$.  Then one lets the column of $\del$ 
indexed by an edge $e$ in $E$ that has 
been oriented from vertex $u$ to $v$ be the difference vector $+u-v$,
regarding each $v$ in $V$ as a standard basis vector for $\RR^{V}$.  
One can regard $\del$ as a map $\ZZ^E \rightarrow \ZZ^V$, and 
define $K(G)$ via either of these equivalent 
presentations (see Proposition~\ref{concordance-prop} below)
\begin{align}
\label{K(G)-Kirchoff-presentation}
K(G)&:= \im \del / \im \del \del^t  \\
\label{K(G)-edge-presentation}     
   & \cong \ZZ^E / \left( \im \del^t + \ker \del \right) 
\end{align}
where $\del^t$ is the  map $\ZZ^V \rightarrow \ZZ^E$ corresponding to the
transpose matrix of $\del$.
The presentation \eqref{K(G)-Kirchoff-presentation} allows one to
compute the structure of $K(G)$ from the
nonzero entries $d_1,d_2,\ldots,d_t$ in the Smith normal form of the {\it graph Laplacian matrix}
$L(G):= \del \del^t$ appearing above: 
$$
K(G) \cong \bigoplus_{i=1}^t \ZZ_{d_i}
$$
where $\ZZ_d:=\ZZ/d\ZZ$ denotes the cyclic group of order $d$.  

\vskip.1in
\noindent
{\bf Example.}
The graphs in the above covering $\tilde{G} \rightarrow G$ 
have node-edge incidence matrices
$$
\begin{aligned}
\del_G&=
\bordermatrix{ 
  & a & b & c & d & e & f & g & h & i \cr
1 &+1 &+1 &+1 &+1 &+1 &+1 & 0 & 0 & 0 \cr
2 &-1 &-1 &-1 &-1 &-1 &-1 & 0 & 0 & 0 
}\\
\del_{\tilde{G}}&=
\bordermatrix{ 
     &a_+&b_+&c_+&a_-&b_-&c_-&d_+&e_+&f_+&d_-&e_-&f_-&g_+&h_+&i_+&g_-& h_- & i_- \cr
1_+  &+1 &+1 &+1 &0  &0  &0  & 0 & 0 & 0 &+1 &+1 &+1 &+1 &+1 &+1 &-1 &-1 &-1 \cr
1_-  &0  &0  &0  &+1 &+1 &+1 &+1 &+1 &+1 &0  &0  &0  &-1 &-1 &-1 &+1 &+1 &+1 \cr
2_+  &-1 &-1 &-1 &0  &0  &0  &-1 &-1 &-1 &0  &0  &0  &0  &0  &0  & 0 & 0 & 0 \cr
2_-  &0  &0  &0  &-1 &-1 &-1 & 0 & 0 & 0 &-1 &-1 &-1 &0  &0  &0  & 0 & 0 & 0 
}
\end{aligned}
$$
from which one obtains the Laplacian matrices
$$
L(G)=\del_G \del_G^t=
\bordermatrix{ 
  & 1 & 2 \cr
1 & 6 & -6 \cr
2 &-6 & 6 \cr
}
\qquad
\text{ and }
\qquad
L(\tilde{G})=\del_{\tilde{G}} \del_{\tilde{G}}^t=
\bordermatrix{ 
    & 1_+ & 1_- & 2_+ & 2_- \cr
1_+ &+12   &-6   &-3  &-3\cr
1_- &-6   &+12   &-3  &-3\cr
2_+ &-3   &-3    &+6   &0 \cr
2_- &-3   &-3    &0     &+6
}
$$
whose Smith normal forms 
$$
\left(
\begin{matrix}
6 & 0 \cr
0 & 0 \cr
\end{matrix}\right)
\qquad
\text{ and }
\qquad
\left(\begin{matrix}
3   &0   &0  &0\cr
0   &3   &0  &0\cr
0   &0    &36   &0 \cr
0   &0    &0     &0
\end{matrix}\right)
$$
allow one to read off their critical groups:
$$
\begin{array}{rll}
K(G)&\cong \ZZ_6 &\cong \ZZ_2 \oplus \ZZ_3,\\
K(\tilde{G})&\cong \ZZ_3^2 \oplus \ZZ_{36} &\cong \ZZ_{4} \oplus \ZZ^2_3 \oplus \ZZ_{9}.
\end{array}
$$

\vskip.1in

Many papers on critical groups have computed examples of $K(G)$ 
via the Smith normal form of $L(G)$.  
On the other hand, there is some literature relating critical groups for different graphs in a {\it functorial} fashion, having
roots in an early paper\footnote{Berman \cite{Berman} worked
not with $K(G)$, but rather the group $B_G(A)$ of 
{\it bicycles over $A$} for $G$ for $A$ an abelian group.  One
can check that $B(G,A)=\Hom_\ZZ(K(G),A)$, that is, $B(G,\QQ/\ZZ)$ is the Pontryagin dual group of $K(G)$, hence isomorphic to $K(G)$.} of Berman \cite{Berman}, work of Lorenzini \cite{Lorenzini}, and unpublished work of Treumann \cite{Treumann}; see \cite{Berget, BMMPR} for some notable applications of this functoriality;  the theory is 
reviewed in Section~\ref{functoriality-subsection} below.  
More recently, the role of critical groups in the analogy between graphs and algebraic curves, discussed originally by Bacher, de la Harpe and Nagnibeda \cite{BHN} (see also Biggs \cite{Biggs}) has been extended further to {\it chip-firing} or {\it sandpile groups} of metric graphs and the theory of tropical geometry, where some of these notions of functoriality appear also in work of Urakawa \cite{Urakawa} and Baker and Norine \cite{BakerNorine1, BakerNorine2} under the name of {\it harmonic morphisms}; for a recent survey and references, 
see Perkinson et al \cite{Perkinson}.

We focus here on the interaction of critical groups with {\it graph coverings}.
The above work of Berman \cite{Berman} and Treumann \cite{Treumann} already showed that
covering maps of graphs $\tilde{G} \overset{\pi} \rightarrow G$
induce surjections $K(\tilde{G}) \rightarrow K(G)$ of their critical groups.
Our goal is to study this surjection, describing its kernel, 
and use this to gain information about $K(\tilde{G})$ from knowledge of $K(G)$.

Section~\ref{coverings-splittings-section} reviews graph coverings,
and proves the easy result that for an $m$-sheeted graph covering 
$\tilde{G} \rightarrow G$, the induced surjection $K(\tilde{G}) \rightarrow K(G)$
splits at all $p$-primary components for primes $p$ that do not divide $m$.

Section~\ref{regular-coverings-section} deals with 
graph coverings which are {\it regular}, in the sense that
$G$ is the quotient $\tilde{G}/H$ for a finite group $H$ acting freely on $\tilde{G}$.
Here one can take advantage of the Gross-Tucker \cite{Gross} encoding of a regular covering 
$\tilde{G} \rightarrow G$ via an $H$-valued {\it voltage assignment} 
$\beta: E \longrightarrow H$ that simply assigns
an arbitrary {\it voltage} $\beta(e)$ in $H$ to each edge $e$ of $G$;
one often calls such extra structure on $G$ 
an {\it $H$-voltage graph} $G_\beta$.
For such voltage graphs $G_\beta$, we will introduce matrices with
coefficients in the {\it group algebra} $\ZZ H$ that allows us
to define (in 
Sections~\ref{regular-coverings-section}, 
\ref{voltage-short-exact-sequence-section})
a notion of {\it voltage graph critical group} $K(G_\beta)$, 
a finite abelian group that naturally
extends the notion of critical group $K(G)$ for graphs.
More importantly, our first main result shows that 
this voltage graph critical group $K(G_\beta)$ fills the role
of presenting the kernel of the 
surjection $K(\tilde{G}) \rightarrow K(G)$.

\begin{theorem}
\label{regular-covering-exact-sequence}
Any $H$-voltage assignment $G_\beta$ with regular covering
$\tilde{G} \overset{\pi}{\rightarrow} G$  has a short exact sequence
$$
0 \rightarrow K(G_{\beta}) \rightarrow K(\tilde{G}) \rightarrow K(G) \rightarrow 0
$$
which splits when restricted to $p$-primary components for primes $p$ not dividing $|H|$.

In particular, the numbers of maximal forests of $\tilde{G}, G$ are related by
a factor of $|K(G_\beta)|$:
$$
|K(\tilde{G})|=|K(G_\beta)| \cdot |K(G)|.
$$
\end{theorem}

As an important special case, 
{\it double (2-sheeted) coverings} $\tilde{G} \rightarrow G$ are always regular,
with $G=\tilde{G}/H$ for the two element group $H=\ZZ_2=\{+,-\}$.  
One can then interpret the $H$-valued voltage assignment on the edges of $G$
as a {\it signed graph} $G_{\pm}$
in the sense of Zaslavsky\footnote{Disallowing half-loops for the moment, although they will be incorporated eventually in Section~\ref{general-signed-graph-section}.} \cite{Zaslavsky}, 
The double cover $\tilde{G} \rightarrow G$
parametrized by a signed graph $G_\pm$ is particularly simple: there
are two vertices $v_+,v_-$ lying above each vertex $v$ of $G$, and
each edge $e=\{u,v\}$ in $G$ gives rise to two edges in $\tilde{G}$, namely
\begin{enumerate}
\item[$\bullet$]
$e_+=\{u_+,v_+\}, e_-=\{u_-,v_-\}$ 
 if $e$  is labelled $+$ in $G_\pm$, and
\item[$\bullet$]
$e_+=\{u_+,v_-\}, e_-=\{u_-,v_+\}$ 
 if $e$  is labelled $-$ in $G_\pm$.
\end{enumerate}
Zaslavsky \cite{Zaslavsky} associated to a signed graph $G_\pm$ an node-edge incidence matrix
$\del=\del_{G_\pm}$ in $\ZZ^{|V| \times |E|}$ generalizing the definition for graphs.  In his $\del$, the column indexed by an edge $e$ in $E$ having positive sign $+$ (resp. negative sign $-$)
that has been oriented from vertex $u$ to $v$ will be the vector $+u-v$ (resp. $+u+v$),
where again one regards each $v$ in $V$ as a standard basis vector in $\RR^{V}$.  
Regarding $\del$ as a map $\ZZ^E \rightarrow \ZZ^V$, as before, one can
define $K(G_\pm)$ via the equivalent presentations
\begin{equation}
\label{signed-K(G)-presentation}
K(G_\pm)= \im \del / \im \del \del^t  
 \cong \ZZ^E / \left( \im \del^t + \ker \del \right) 
\end{equation}
where $\del^t$ is the transpose matrix considered as a map $\ZZ^V \rightarrow \ZZ^E$.
The {\it signed graph Laplacian matrix} $L(G_\pm):=\del \del^t$ appearing here already figured into Zaslavsky's signed version of the matrix tree theorem \cite[Thm. 8.A.4]{Zaslavsky}, allowing us to interpret the cardinality $|K(G_\pm)|$
as a weighted count of objects that one can think of as {\it maximal forests} in $G_\pm$;
see Section~\ref{signed-matrix-tree-subsection} below.  
Theorem~\ref{regular-covering-exact-sequence} then specializes as follows.

\begin{theorem}
\label{unsigned-double-cover-theorem}
For each signed graph $G_{\pm}$, parametrizing a graph double covering
$\tilde{G} \rightarrow G$, one has a short exact sequence of critical groups
$$
0 \rightarrow K(G_\pm) \rightarrow K(\tilde{G}) \rightarrow K(G) \rightarrow 0,
$$
splitting on restriction to $p$-primary components for odd primes $p$.
In particular, $|K(\tilde{G})| = |K(G_\pm)| \cdot |K(G)|$.
\end{theorem}

\vskip.1in
\noindent
{\bf Example.}
Our earlier double covering $\tilde{G} \rightarrow G$ is parametrized by
this signed graph $G_{\pm}$:

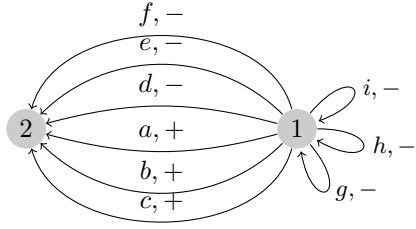
\begin{center}
\begin{tikzpicture}[->,scale=1.8, auto,swap]
    \foreach \pos/\name in {{(0,0)/2},{(2,0)/1}}
        \node[vertex] (\name) at \pos {$\name$};
    \foreach \source/ \dest/\arclabel/\voltage in {1/2/a/+}
        \path (\source) edge[bend left=15] node[auto] {$\arclabel, \voltage$} (\dest);
    \foreach \source/ \dest/\arclabel/\voltage in {1/2/b/+}
        \path (\source) edge[bend left=45] node[auto] {$\arclabel, \voltage$} (\dest);
    \foreach \source/ \dest/\arclabel/\voltage in {1/2/c/+}
        \path (\source) edge[bend left=75] node[auto] {$\arclabel, \voltage$} (\dest);
    \foreach \source/ \dest/\arclabel/\voltage in {1/2/d/-}
        \path (\source) edge[bend right=15] node[auto] {$\arclabel, \voltage$} (\dest);
    \foreach \source/ \dest/\arclabel/\voltage in {1/2/e/-}
        \path (\source) edge[bend right=45] node[auto] {$\arclabel, \voltage$} (\dest);
    \foreach \source/ \dest/\arclabel/\voltage in {1/2/f/-}
        \path (\source) edge[bend right=75] node[auto] {$\arclabel, \voltage$} (\dest);

\foreach \source in {1}    
	\path[every node/.style={font=\sffamily\small}]
        (\source)   edge[in=20, out=50, loop] node[right] {$i,-$}  (\source);
\foreach \source in {1}    
	\path[every node/.style={font=\sffamily\small}]
        (\source)   edge[in=-30, out=0, loop] node[right] {$h,-$}  (\source);
\foreach \source in {1}    
	\path[every node/.style={font=\sffamily\small}]
        (\source)   edge[in=-80, out=-50, loop] node[right] {$g,-$}  (\source);
\end{tikzpicture}
\end{center}
having node-edge incidence matrix $\del$ and Laplacian matrix 
$$
\del=
\begin{aligned}
&\bordermatrix{ 
  & a & b & c & d & e & f & g & h & i \cr
1 &+1 &+1 &+1 &+1 &+1 &+1 &+2 &+2 &+2 \cr
2 &-1 &-1 &-1 &+1 &+1 &+1 & 0 & 0 & 0 
} &\text{ and }
L(G)&=\del \del^t=
\bordermatrix{ 
  & 1 & 2 \cr
1 & 18 & 0 \cr
2 & 0 & 6 \cr
}.
\end{aligned}
$$
Consequently
$$
\ZZ^2/\im \del \del^t 
\quad \cong \quad 
\ZZ_6 \oplus \ZZ_{18} 
\quad \cong \quad 
\ZZ_2^2 \oplus \ZZ_3 \oplus \ZZ_{9}.
$$
One can check that $\im \del$ here is the
sublattice $\ZZ^2_{\equiv 0 \bmod{2}}$
of index $2$ inside $\ZZ^2$ where the sum of coordinates is even.
Therefore $K(G_\pm)$ is an index $2$ subgroup of $\ZZ^2/\im \del \del^t$.
Thus the answer for $\ZZ^2/\im \del \del^t$ given above forces
$$
K(G_{\pm}) 
\quad \cong \quad
\im \del/\im \del \del^t 
\quad \cong \quad 
\ZZ_2 \oplus \ZZ_3 \oplus \ZZ_{9},
$$
and the short exact sequence from Theorem~\ref{unsigned-double-cover-theorem} takes this form:
$$
\begin{array}{rcccccl}
0 \rightarrow & K(G_\pm) & \rightarrow & K(\tilde{G}) & \rightarrow & K(G) & \rightarrow 0 \\
              & \Vert&             & \Vert        &             & \Vert  & \\
              & \ZZ_2 \oplus \ZZ_3 \oplus \ZZ_{9}& & \ZZ_{4} \oplus \ZZ^2_3 \oplus \ZZ_{9} &   & \ZZ_2 \oplus \ZZ_3  &
\end{array}
$$
Note that its $p$-primary component splits at the odd prime $p=3$
$$
0 \rightarrow \ZZ_3 \oplus \ZZ_9 
   \rightarrow \ZZ_3^2 \oplus \ZZ_9 
    \rightarrow \ZZ_3 
      \rightarrow 0
$$
but does {\it not split} at the prime $p=2$
$$
0 \rightarrow \ZZ_2 \rightarrow \ZZ_4 \rightarrow \ZZ_2 \rightarrow 0.
$$
\vskip.2in

Having developing this theory in the earlier sections,
Section~\ref{diagonal-application-section} describes a 
class of nontrivial examples of regular graph coverings
where the theory is particularly easy to apply, because the
relevant voltage graph critical group $K(G_\beta)$ has the peculiar property
that its presentation involves a {\it diagonal} Laplacian matrix! 

Sections~\ref{general-signed-graph-section},
\ref{signed-graph-double-cover-section},
\ref{short-complex-section} 
return to the special case of
signed graphs, but generalize 
Theorem~\ref{unsigned-double-cover-theorem} in a different direction
than Theorem~\ref{regular-covering-exact-sequence}
The idea is to allow {\it half-loops} as in Zaslavsky's
original paper \cite{Zaslavsky}, and also to introduce
a notion of {\it double-covering of signed graphs} in which all three
players involved (the base, the cover, the voltage assignment) are
{\it signed graphs}.  
This allows us to prove a more flexible
double covering result, Theorem~\ref{signed-graph-double-cover-complex},
which we apply to two more families of examples 
in Section~\ref{signed-graph-double-cover-application-section}.

\section{Review of critical groups}
\label{critical-group-review-section}

\subsection{Presentations of the critical group}

Given a multigraph $G=(V,E)$, as mentioned in the Introduction, we will
let $\ZZ^E, \ZZ^V$ have $\ZZ$-bases indexed by $E, V$, and then fix an orientation $e=(u,v)$ for each edge $e$ in $E$, so as to define
the $\ZZ$-linear {\it node-edge} map and incidence matrix via
$$
\begin{array}{rcl}
\ZZ^E & \overset{\del}{\longrightarrow}& \ZZ^V \\
e & \longmapsto & +u - v
\end{array}
$$
Then $\del=\del_G=\del_\ZZ$ is represented by a matrix in $\ZZ^{V \times E}$
with respect to these bases.  One will also sometimes want to think of
the associated $\RR$-linear 
map $\RR^E  \overset{\del_{\RR}}{\longrightarrow} \RR^V$.

Choosing inner products 
$\langle \cdot, \cdot \rangle_E, \langle \cdot, \cdot \rangle_V$
on $\RR^E, \RR^V$ that make the above bases each orthonormal, the
{\it transpose matrix} $\del^t$ represents the {\it adjoint map}
$\ZZ^V  \overset{\del^t}{\longrightarrow} \ZZ^E$
or 
$\RR^V  \overset{\del_\RR^t}{\longrightarrow} \RR^E$
defined by 
$$
\langle \del x, y \rangle_V
=\langle x, \del^ty \rangle_E 
$$
for all $x$ in $\ZZ^E$ and $y$ in $\ZZ^V$. 
It is easily seen that
\begin{equation}
\label{rational-orthogonal-decomposition}
\RR^E = \im \del_\RR^t \oplus \ker \del_\RR
\end{equation}
is an orthogonal direct sum decomposition.
Here $\ker \del_\RR \subset \RR^E$ and 
$\ker \del \subset \ZZ^E$ are called the {\it cycle space} and
{\it cycle lattice} of $G$, while 
 $\im \del^t_\RR \subset \RR^E$ and 
$\im \del^t \subset \ZZ^E$ are called the {\it bond or cut space} and
{\it bond or cut lattice} of $G$; see \cite{BHN, Biggs, GodsilRoyle}.  The critical group $K(G)$ 
can be thought of as measuring the failure of 
equality in \eqref{rational-orthogonal-decomposition} when working with 
the lattices instead of the $\RR$-linear spaces that they span.

\begin{definition}
\label{edge-presentation}
Define the {\it critical group} 
$$
K(G):= \ZZ^E / \left( \im \del^t + \ker \del \right).
$$
\end{definition}

\noindent
The agreement between the two presentations of $K(G)$ given in \eqref{K(G)-Kirchoff-presentation}
and \eqref{K(G)-edge-presentation}, as well as the two 
presentations of $K(G_\pm)$) in \eqref{signed-K(G)-presentation},
is explained by the following.

\begin{proposition}
\label{concordance-prop}
Given abelian group homomorphisms 
$A \underset{g}{\overset{f}{\rightleftarrows}} B$,
the map $f$ induces an isomorphism
$$
A/ \left( \im g + \ker f \right) \longrightarrow \im f / \im fg.
$$
In particular, applying this to 
$\ZZ^E \underset{\del^t}{\overset{\del}{\rightleftarrows}} \ZZ^V$,
the map $\del$ induces an isomorphism
$$
K(G) = \ZZ^E /  \left( \im \del^t + \ker \del \right)
 \longrightarrow \im \del / \im \del \del^t.
$$
\end{proposition}
\begin{proof}
The composite of two surjections
$
A \overset{f}{\twoheadrightarrow} \im f \twoheadrightarrow \im f / \im fg
$
annihilates both $\ker f$ and $\im g$, inducing a
surjection $A/ \left( \im g + \ker f \right) \twoheadrightarrow \im f / \im fg$.
To see that this surjection is also injective, note that for 
$a$ in $A$ to represent an element in the kernel of this surjection means
that $f(a)$ lies in $\im fg$, so that $f(a)=f(g(b))$ for some $b$ in $B$. 
This means $a-g(b)$ lies in $\ker(f)$, and hence the expression
$a = g(b) + (a-g(b))$ shows that $a$ represents the zero coset of 
$A/ \left( \im g + \ker f \right)$.
\end{proof}

\subsection{Functoriality and Pontryagin duality}
\label{functoriality-subsection}

Given multigraphs $G_i=(V_i,E_i)$, for $i=1,2$, and
a $\ZZ$-linear map $\ZZ^{E_1} \overset{f}{\rightarrow} \ZZ^{E_2}$ satisfying 
\begin{equation}
\label{functoriality-assumptions}
\begin{aligned}
f(\im \del_{G_1}^t) & \subset \im \del_{G_2}^t\\
f(\ker \del_{G_1}) & \subset \ker \del_{G_2},
\end{aligned}
\end{equation}
the presentation of $K(G)$ in Definition~\ref{edge-presentation}
shows that $f$ induces a homomorphism $K(G_1) \overset{f}{\longrightarrow} K(G_2)$.
Such homomorphisms will be our fundamental tools.

An important feature in this setting is the fact 
(\cite[Proposition 2.3]{BMMPR}, \cite[Proposition 5]{Treumann}) that
the assumptions \eqref{functoriality-assumptions} are closed
under taking adjoints/transposes:   the adjoint map
$\ZZ^{E_2} \overset{f^t}{\rightarrow} \ZZ^{E_1}$ will also satisfy
$$
\begin{aligned}
f^t(\im \del_{G_2}^t) & \subset \im \del_{G_1}^t\\
f^t(\ker \del_{G_2}) & \subset \ker \del_{G_1}\\
\end{aligned}
$$

\subsection{Pontryagin duality}
\label{Pontryagin-duality-section}

Given the homomorphism $K(G_1) \overset{f}{\rightarrow} K(G_2)$
induced by a map $\ZZ^{E_1} \overset{f}{\rightarrow} \ZZ^{E_2}$,
we will often wish to apply the {\it Pontryagin duality isomorphism} 
\begin{equation}
\label{Pontryagin-dual}
K(G) \cong \hat{K}(G):=\Hom(K(G),\QQ/\ZZ)
\end{equation}
to both of the finite abelian groups $K(G_i)$, and instead consider
the dual morphism
$
\hat{K}(G_2) \overset{\hat{f}}{\rightarrow} \hat{K}(G_1).
$
In the case of critical groups, the isomorphism \eqref{Pontryagin-dual}
is very natural.

\begin{proposition}(\cite[Prop. 2.5]{BMMPR}, \cite[Prop. 9]{Treumann})
\label{Pontryagin-duality-prop}
For multigraphs $G$, the Pontryagin duality isomorphism 
in \eqref{Pontryagin-dual}
can be chosen\footnote{Although not needed, the 
description of the isomorphism in \eqref{Pontryagin-dual}
is as follows.  Letting
$
\pi_{\im \del_\RR^t}: \RR^E \longrightarrow \im \del_\RR^t
$
be {\it orthogonal projection} onto the bond space, send an element of
$K(G)=\ZZ^E /\left( \im \del^t + \ker \del \right)$ 
represented by $x$ in $\ZZ^E$ to the homomorphism $K(G) \rightarrow \QQ/\ZZ$
which maps an element of $K(G)$ represented by $y$ to the additive coset
$\langle \pi_{\im \del_\RR^t} (x), y  \rangle_E +\ZZ$ in $\QQ/\ZZ$.}
so that any map $\ZZ^{E_1} \overset{f}{\rightarrow} \ZZ^{E_2}$ 
satisfying the assumptions \eqref{functoriality-assumptions}
will make the following diagram commute:

\begin{equation*}
\begin{CD}
   K(G_2)   @>{f^t}>>  K(G_1)  \\
    @V{\wr}VV                              @VV{\wr}V     \\
    \hat{K}(G_2)     @>{\hat{f}}>>              \hat{K}(G_1)    
\end{CD}
\end{equation*}
\end{proposition}

\section{Graph coverings, surjections,  and splittings}
\label{coverings-splittings-section}

Here we recall the notion of a graph covering as in Gross and Tucker \cite[\S 2]{Gross}, 
and then prove a refinement of results of Treumann \cite{Treumann} 
and of Baker and Norine \cite{BakerNorine2} for coverings.

\begin{definition}
Given two multigraphs $\tilde{G}=(\tilde{V},\tilde{E})$ and $G=(V,E)$ 
a {\it graph map} is a continuous map $\tilde{G} \overset{\pi}{\rightarrow} G$ 
of their underlying topological spaces that maps the 
interior of each edge of $\tilde{G}$ homeomorphically onto the 
interior of some edge of $G$.  
\end{definition}

In particular, a graph map $\pi$ induces
a set map $\tilde{E} \overset{\pi}{\rightarrow} E$;  considering
what happens via continuity at the endpoints of each edge, it also
induces a set map $\tilde{V} \overset{\pi}{\longrightarrow} V$.
Note that when one has a graph map
$\tilde{G} \overset{\pi}{\rightarrow} G$,
any orientation of the edges of $G$ pulls back to a compatible orientation of
the edges of $\tilde{G}$ in such a way that $f$ preserves
orientation.  Henceforth we will always assume that $\tilde{G}, G$ are oriented
compatibly in this fashion when
writing down node-edge incidence matrices $\del_{\tilde{G}}, \del_G$.

\begin{definition}
Say that a graph map $\tilde{G} \overset{\pi}{\rightarrow} G$ 
is a {\it graph covering} if every vertex of $\tilde{G}$ has
a neighborhood on which the restriction of $\pi$ is a homeomorphism.
\end{definition}

It is not hard see that within a fixed connected component of the base graph $G$,
every vertex $v$ and edge $e$ will have
the same cardinality $m$ for the inverse image sets $\pi^{-1}(v), \pi^{-1}(e)$.

\begin{definition}
Say $\tilde{G} \overset{\pi}{\rightarrow} G$ is an {\it $m$-sheeted cover}
if $|\pi^{-1}(v)|=|\pi^{-1}(e)|=m$ for every component of $G$.
\end{definition}

We come now to the main observation of this section.

\begin{proposition}(cf. Baker-Norine \cite[\S 4]{BakerNorine2}, 
Berman \cite[Thm. 5.7]{Berman}, 
Treumann \cite[Prop. 19]{Treumann})
\label{covering-splitting-proposition}
An $m$-sheeted covering $\tilde{G} \overset{\pi}{\rightarrow} G$ of finite
multigraphs gives rise to a surjection of critical groups
$K(\tilde{G}) \overset{\pi}{\twoheadrightarrow} K(G)$.  

Furthermore, the backward map $K(G) \overset{\pi^t}{\rightarrow} K(\tilde{G})$
satisfies $\pi \pi^t = m \cdot 1_{K(G)}$, and hence splits off the $p$-primary component of $K(G)$ as a direct summand for each prime $p$ that does not divide $m$.
\end{proposition}
\begin{proof}
We first need to check that 
$\pi$ satisfies the two conditions \eqref{functoriality-assumptions}.

For the first condition, note that for any graph map 
$\tilde{G} \overset{\pi}{\rightarrow} G$, the associated set maps 
$\tilde{E} \overset{\pi}{\rightarrow} E$ and
$\tilde{V} \overset{\pi}{\rightarrow} V$ induce a chain map, that is,
one has a commutative square
$$
\begin{CD}
\ZZ^{\tilde{E}} @>\del_{\tilde{G}}>> \ZZ^{\tilde{V}}\\
@V{\pi}VV @VV{\pi}V \\
\ZZ^{E} @>\del_G>> \ZZ^V.
\end{CD}
$$
Consequently, the left vertical map $\pi$ in this square sends
sends (oriented) cycles of $\tilde{G}$ to cycles of $G$, that is,
$\pi(\ker \del_{\tilde{G}}) \subset \ker \del_G$.


For the second condition, note that when $\pi$ is not just a graph map but a graph covering, our conventions for inducing orientations of edges in
$\tilde{E}$ from orientations in $E$ lead to a similar commutative square
$$
\begin{CD}
\ZZ^{\tilde{E}} @<\del^t_{\tilde{G}}<< \ZZ^{\tilde{V}}\\
@V{\pi}VV @VV{\pi}V \\
\ZZ^{E} @<\del^t_G<< \ZZ^V.
\end{CD}
$$
Consequently the left vertical map $\pi$ similarly 
satisfies $\pi(\im \del_{\tilde{G}}^t) \subset \im \del_G^t$, as desired. 

Thus $\pi$ induces a map $K(\tilde{G}) \rightarrow K(G)$.  It is surjective
because $\ZZ^{\tilde{E}} \overset{\pi}{\rightarrow} \ZZ^E$ is already surjective.

The assertion $\pi \pi^t = m \cdot 1_{K(G)}$ for the induced maps on critical groups
follows because the same holds on the level of $\ZZ^E$:
one has $\pi \pi^t = m \cdot 1_{\ZZ^E}$ because every edge $e$ of $G$ 
has exactly $m$ preimages in $\pi^{-1}(e) \subset \tilde{E}$.
\end{proof}

\begin{remark}
Both Berman and Treumann considered a situation somewhat more 
general than a covering that leads to a surjection of critical groups.
Berman \cite[p.9]{Berman} 
defined what it means for a graph $\tilde{G}$ to be {\it divisible by} $G$, 
leading to a graph map 
$\tilde{G} \rightarrow G$ which Treumann \cite[Definition 16]{Treumann}
called a {\it Berman bundle}.  
Most of our results can be made to work, with extra technicality, at the level of Berman bundles; see the
second author's REU report \cite{Tseng}.  We have not yet found sufficiently
interesting applications requiring this extra level of generality, and so we suppress this discussion here.  

Similarly, Baker and Norine consider
{\it harmonic maps} which are more general than 
coverings \cite[\S2, Example 3.4]{BakerNorine2}, showing that
the assertions of Proposition~\ref{covering-splitting-proposition}
hold (with modified statements) in that setting; see their 
Lemmas 4.1 and Lemma 4.12, and their Theorem 4.13.
\end{remark}

\section{Regular coverings and voltage graphs}
\label{regular-coverings-section}

We recall here the notion of regular graph coverings from Gross and Tucker \cite[\S 1]{Gross}.

\begin{definition}
\label{regular-cover-definition}
For a multigraph $G$, a graph map $G \overset{h}{\rightarrow} G$ is called 
a {\it graph endomorphism}.  If it has an inverse $G \overset{h^{-1}}{\rightarrow} G$
which is also a graph endomorphism, then $h$ is called a {\it graph automorphism}.

Say that a group $H$ {\it acts on the right on} $G$ if every $h$ in $H$ corresponds to a graph
automorphism of $G$, 
in such a way that $h_1 (h_2(x)) = (h_2 h_1)(x)$ for all $h_1, h_2$ in $H$ and
all edges $e$ of $G$.

Say a graph covering $\tilde{G} \overset{\pi}{\rightarrow} G$ is
{\it regular} (or {\it normal} or {\it Galois}) if there exists a group $H$ acting on the right on $\tilde{G}$
with the property that $H$ acts {\it simply transitively} on all fibers $\pi^{-1}(v)$ and $\pi^{-1}(e)$ for every vertex $v$ and edge $e$ of $G$.  
In this situation, $H$ is called the {\it transformation group}
of the regular covering $\tilde{G} \overset{\pi}{\rightarrow} G$.
\end{definition}

\begin{remark}
An alternative way to phrase a regular covering is to say that there is a group $H$ acting via cellular automorphisms
on the cell complex $\tilde{G}$, with the action being {\it free} on the associated topological space.  Then
$\tilde{G} \overset{\pi}{\rightarrow} \tilde{G}/H:=G$ is the quotient mapping;  see Gross and Tucker \cite[Thm. 4]{Gross}.
\end{remark}

We next review the encoding from \cite[\S 4]{Gross} of 
a regular graph covering $\tilde{G} \overset{\pi}{\rightarrow} G$ 
with transformation group $H$ and base graph $G=(V,E)$, via an  
{\it $H$-voltage assignment} or {\it $H$-voltage graph}
$G_\beta$, which is nothing more than a set map 
$
\beta: E \rightarrow H.
$

\vskip.1in
\noindent
{\sf From a regular covering to a voltage assignment.}

Given a regular graph covering $\tilde{G} \overset{\pi}{\rightarrow} G$, 
arbitrarily choose for each vertex $v$ in $V$
one vertex $v_\one$ in $\pi^{-1}(v)$ to be labelled by the identity element $\one$ of
$H$.  Since $H$ acts simply transitively on $\pi^{-1}(v)$, the remaining elements in the fiber
can be labelled uniquely as $v_h:=h(v_\one)$.  Since $H$ acts on the right, this forces that 
\begin{equation}
\label{right-action-in-vertex-fiber}
h_1(v_{h_2}) = h_1(h_2(v_{\one})) = (h_2 h_1)(v_\one) = v_{h_2 h_1}.
\end{equation}
To get the voltage assignment $\beta(e)$ for an edge $e$ in $E$, first assume that the
orientation of $e=(u,v)$ has been pulled back to all of the edges in the fiber $\pi^{-1}(e)$.
There will be a unique such edge $e_\one$ having source $u_\one$;  if this edge has target
$v_h$, then decree that $\beta(e)=h$.  Since $H$ acts by automorphisms, one can use this to label
the remaining edges in the same fiber:  for any $h'$ in $H$ 
the edge $e_{h'}:=h'(e_\one)$ must have source $h'(u_\one)=u_{h'}$ and target $h'(v_h)=v_{h h'}$.
In other words, the edges of $\tilde{E}$ in $\pi^{-1}(e)$ are all of the form 
$e_{h'}=(u_{h'},v_{\beta(e)h'})$ as $h'$ ranges through $H$, and the $H$-action on them follows this rule:
\begin{equation}
\label{right-action-in-edge-fiber}
h_1(e_{h_2}) = e_{h_2 h_1}.
\end{equation}

\vskip.1in
\noindent
{\sf From a voltage assignment to a regular covering.}

Given a multigraph $G=(V,E)$, with an arbitrary orientation on $E$,
and an arbitrary $H$-voltage assignment $G_\beta$ as a map $\beta: E\rightarrow H$, 
one creates $\tilde{G}=(\tilde{E},\tilde{V})$ 
as follows:
$$
\begin{aligned}
\tilde{V}&:= \{v_h\}_{v \in V, h \in H}\\
\tilde{E}&:= \{e_h=(u_h, v_{\beta(e)h})\}_{e=(u,v) \in E,  h \in H}
\end{aligned}
$$
The regular graph covering $\tilde{G} \overset{\pi}{\rightarrow} G$ simply forgets the subscripts:
$e_h \mapsto e$ and $v_h \mapsto v$.

\begin{example}
\label{octahedron-example}
Let $\tilde{G}$ be the graph of the octahedron, which carries a free action of the cyclic group $H=\{1,h,h^2\} \cong \ZZ_3$, in which $h$ rotates $120^\circ$ around an axis passing through the centers of two opposite triangular faces.  One finds that the associated regular covering $\tilde{G} \overset{\pi}{\rightarrow} G$ is as shown below, described by a voltage graph $G_\beta$ 
on an underlying multigraph 
$G=(V,E)$ with two vertices $V=\{u,v\}$ and four edges $E=\{a,b,c,d\}$.  
Here edges $a,b$ are both directed from $u$ to $v$ 
while $c,d$ are loops on vertices $u,v$, respectively, with voltage assignments
$\beta(a)=\one, \beta(b)=\beta(c)=\beta(d)=h$.

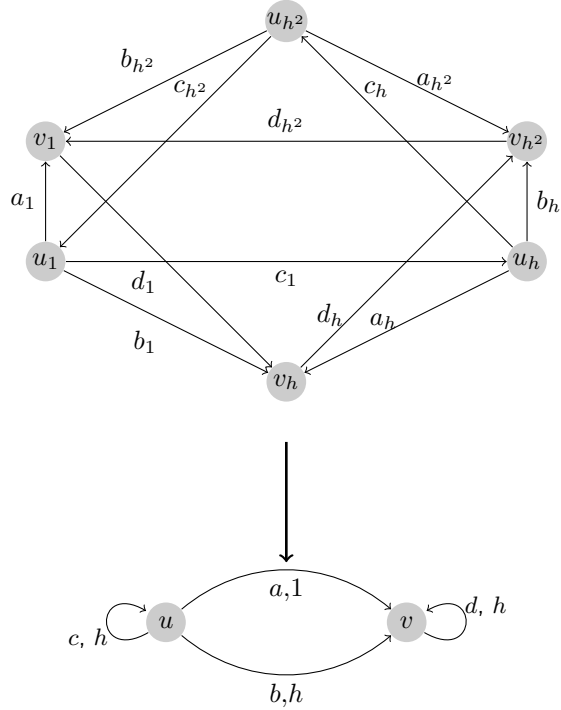
\begin{center}
\begin{tikzpicture}[->,scale=0.8, auto,swap]
    \foreach \pos/\name in {{(3,1)/v_{h^2}}, {(3,-1)/u_{h}},{(-5,1)/v_{1}},
                           {(-5,-1)/u_{1}},
                           {(-1,3)/u_{h^2}},
                           {(-1,-3)/v_{h}},
                           {(1,-7)/v},
                           {(-3,-7)/u}}
        \node[vertex] (\name) at \pos {$\name$};


    \foreach \source/ \dest/\arclabel in {u_{1}/v_{1}/a_1}
        \path (\source) edge node[left] {$\arclabel$} (\dest);
    \foreach \source/ \dest/\arclabel in {u_{h}/v_{h}/a_h}
        \path (\source) edge node[left] {$\arclabel$} (\dest);
    \foreach \source/ \dest/\arclabel in {u_{h^2}/v_{h^2}/a_{h^2}}
        \path (\source) edge node[right] {$\arclabel$} (\dest);

    \foreach \source/ \dest/\arclabel in {u_{1}/v_{h}/b_1}
        \path (\source) edge node[auto] {$\arclabel$} (\dest);
    \foreach \source/ \dest/\arclabel in {u_{h}/v_{h^2}/b_h}
        \path (\source) edge node[auto] {$\arclabel$} (\dest);
    \foreach \source/ \dest/\arclabel in {u_{h^2}/v_{1}/b_{h^2}}
        \path (\source) edge node[auto] {$\arclabel$} (\dest);

    \foreach \source/ \dest/\arclabel in {u_{1}/u_{h}/c_1}
        \path (\source) edge node[auto] {$\arclabel$} (\dest);
    \foreach \source/ \dest/\arclabel in {u_{h}/u_{h^2}/c_h}
        \path (\source) edge node[near end, right] {$\arclabel$} (\dest);
    \foreach \source/ \dest/\arclabel in {u_{h^2}/u_{1}/c_{h^2}}
        \path (\source) edge node[near start, left] {$\arclabel$} (\dest);

    \foreach \source/ \dest/\arclabel in {v_{1}/v_{h}/d_1}
        \path (\source) edge node[auto] {$\arclabel$} (\dest);
    \foreach \source/ \dest/\arclabel in {v_{h}/v_{h^2}/d_h}
        \path (\source) edge node[near start, left] {$\arclabel$} (\dest);
    \foreach \source/ \dest/\arclabel in {v_{h^2}/v_{1}/d_{h^2}}
        \path (\source) edge node[auto] {$\arclabel$} (\dest);

    \foreach \source/ \dest/\arclabel/\voltage in {u/v/a/1}
        \path (\source) edge[bend left=40] node[auto] {$\arclabel$,$\voltage$} (\dest);
    \foreach \source/ \dest/\arclabel/\voltage in {u/v/b/h}
        \path (\source) edge[bend right=40] node[auto] {$\arclabel$,$\voltage$} (\dest);
    \foreach \source in {v}    
	\path[every node/.style={font=\sffamily\small}]
        (\source)   edge[in=30, out=330, loop] node[right, near end] {$d$, $h$}  (\source);
    \foreach \source in {u}    
	\path[every node/.style={font=\sffamily\small}]
        (\source)   edge[in=150, out=210, loop] node[left, near start] {$c$, $h$}  (\source);

\draw[arrows=->,line width=1pt](-1,-4)--(-1,-6);
\end{tikzpicture}
\end{center}

\end{example}

\begin{example}
\label{double-covers-are-regular-example}
Call a graph covering $\tilde{G} \overset{\pi}{\rightarrow} G$ a {\it double cover} if
it is $2$-sheeted.  
We claim that graph double covers are {\it always} regular, with transformation group $H=\ZZ_2 =\{+,-\}$:  picking an arbitrary labelling of the two vertices in each fiber
$\pi^{-1}(v) =\{v_+, v_-\}$ and $\pi^{-1}(e) = \{e_+, e_-\}$, one finds that the involution $h$ which simultaneously swaps all $v_+ \leftrightarrow v_-$ and $e_+ \leftrightarrow e_-$ is
a graph automorphism generating the transformation group 
$H=\{\one, h\} \cong \ZZ_2$ that satisfies the Definition~\ref{regular-cover-definition}
for a regular covering.
In this setting, the voltage assignment $G_\beta$ as a function $E \rightarrow H=\ZZ_2=\{+,-\}$
can be thought as a signed graph $G_\pm$ as in the Introduction.
\end{example}

We can now use this $H$-voltage assignment encoding of regular coverings
to reformulate the critical group $K(\tilde{G})$ using the group algebra of $H$.
This reformulation will be useful in the proof of
Theorem~\ref{regular-covering-exact-sequence} below.

\begin{definition}
Recall that the {\it group algebra $\ZZ H$} is the 
free $\ZZ$-module on $\ZZ$-basis elements $\{T_h\}_{h \in H}$
with multiplication defined $\ZZ$-linearly via $T_{h_1} T_{h_2} := T_{h_1 h_2}$.
\end{definition}

For any $H$-voltage assignment $G_\beta$ and associated regular covering
$\tilde{G} \overset{\pi}{\rightarrow} G$, the action of $H$ on the right of
$\tilde{G}=(\tilde{V},\tilde{E})$ endows $\ZZ^{\tilde{E}}$ and $\ZZ^{\tilde{V}}$
with the structures of right-$\ZZ H$-modules:
$$
\begin{aligned}
e_{h_1} T_{h_2}&:= e_{h_1 h_2} \\
v_{h_1} T_{h_2}&:= v_{h_1 h_2} 
\end{aligned}
$$

We will also work with {\it free} right-$\ZZ H$-modules 
$(\ZZ H)^E$ and $(\ZZ H)^V$ having $\ZZ H$-basis elements indexed by
$e$ in $E$ and $v$ in $V$.  This means, for example, that
$(\ZZ H)^E$ is a free $\ZZ$-module
with $\ZZ$-basis elements $\{e T_h\}_{e \in E, h \in H}$, and
and its right-$\ZZ H$-module structure can be defined $\ZZ$-linearly by
$$
(e T_{h_1}) T_{h_2}:= e T_{h_1 h_2}
$$

\begin{proposition}
\label{group-algebra-covering-identifications}
For any $H$-voltage assignment $G_\beta$ and associated regular covering
$\tilde{G} \overset{\pi}{\rightarrow} G$, the following $\ZZ$-module maps give
isomorphisms of right-$\ZZ H$-modules:
$$
\begin{array}{rcl}
\ZZ^{\tilde{E}} &\longrightarrow &(\ZZ H)^E \\
e_h & \longmapsto & e T_h \\
& & \\
\ZZ^{\tilde{V}} &\longrightarrow &(\ZZ H)^V \\
v_h & \longmapsto & v T_h .
\end{array}
$$
\end{proposition}
\begin{proof}
This follows from the fact that we have labelled 
the elements within the fibers $\pi^{-1}(v)=\{v_h\}_{h \in H}$ 
and $\pi^{-1}(e)=\{e_h\}_{h \in H}$ in such a way that
the right-$H$-actions satisfy the rules
\eqref{right-action-in-vertex-fiber} and \eqref{right-action-in-edge-fiber}.
\end{proof}

Here is the point of working with right-actions and right-$\ZZ H$-modules:
\begin{itemize}
\item
one can regard elements of $(\ZZ H)^E$ and $(\ZZ H)^V$ as column vectors
having entries in $\ZZ H$, and then
\item  
specify {\it right}-$\ZZ H$-module maps between these
free right-$\ZZ H$-modules via multiplication
on the {\it left} by matrices with entries in $\ZZ H$.
\end{itemize}
For example, define $\del_{G_\beta}$ to be the matrix in $(\ZZ H)^{V \times E}$ 
representing the right-$\ZZ H$-module map 
\begin{equation}
\label{group-algebra-del-definition}
\begin{array}{rcll}
(\ZZ H)^E &\overset{\del_{G_\beta}}{\longrightarrow}& (\ZZ H)^V &\\
e &\longmapsto &+u-v T_{\beta(e)} &\\
e T_h &\longmapsto &(+u-v T_{\beta(e)}) T_h &= +u T_h - v T_{\beta(e) h}\\
\end{array}
\end{equation}
for each edge $e$ of $G$ which is oriented $e=(u,v)$.
We will also need a map in the other direction
$$
(\ZZ H)^V \overset{\del^*_{G_\beta}}{\longrightarrow} (\ZZ H)^E
$$
which is represented by the matrix $\del^*_{G_\beta}$ in $(\ZZ H)^{E \times V}$
obtained from $\del_{G_\beta}$ by first transposing the matrix, and then applying to
each $\ZZ H$-entry the anti-automorphism $\ZZ H \overset{*}{\rightarrow} \ZZ H$ 
sending $T_h \mapsto T_{h^{-1}}$.

\begin{proposition}
\label{group-algebra-isomorphisms}
The isomorphisms in Proposition~\ref{group-algebra-covering-identifications}
make the following diagrams of right-$\ZZ H$-module morphisms commute:
$$
\begin{CD}
\ZZ^{\tilde{E}} @>>> (\ZZ H)^E \\
@V{\del_{\tilde{G}}}VV  @VV{\del_{G_\beta}}V \\
\ZZ^{\tilde{V}} @>>> (\ZZ H)^V \\
\end{CD}
\qquad \qquad
\begin{CD}
\ZZ^{\tilde{V}} @>>> (\ZZ H)^V \\
@V{\del^t_{\tilde{G}}}VV   @VV{\del^*_{G_\beta}}V \\
\ZZ^{\tilde{E}} @>>> (\ZZ H)^E \\
\end{CD}
$$
\end{proposition}
\begin{proof}
To see the commutativity of the left diagram, note that
the basis element of $\ZZ^{\tilde{E}}$ corresponding to a 
directed edge $e_h = (u_h,v_{\beta(e)h})$ in $\tilde{E}$,
lying above $\pi(e_h)=e=(u,v)$ in $E$, will map
under $\del_{\tilde{G}}$ to the vector $+u_h - v_{\beta(e)h}$ in $\ZZ^{\tilde{V}}$.
Since the horizontal isomorphisms send $e_h \mapsto e T_h$ and 
$+u_h - v_{\beta(e)h} \mapsto +u T_h - v T_{\beta(e)h}$, commutativity
follows from the last line of \eqref{group-algebra-del-definition}.

The commutativity of the right diagram then follows from a general fact:
the horizontal isomorphisms carry the inner products 
on $\ZZ^{\tilde{V}}, \ZZ^{\tilde{E}}$ to inner products on 
$(\ZZ H)^V, (\ZZ H)^E$ that make
$\{ v T_h \}_{v \in V, h \in H}$ and 
$\{ e T_h \}_{e \in E, h \in H}$
orthonormal bases.  This implies that a right-$\ZZ H$-module map
$(\ZZ H)^E \rightarrow (\ZZ H)^V$ represented by a matrix $M=(m_{v,e})$ in 
$(\ZZ H)^{V \times E}$, has its adjoint map represented by
the matrix $M^*$ in the above notation.  To check this,
write $m_{v,e}=\sum_{ h \in H} \mu_{v,e,h} T_h$ for some
$\mu_{v,e,h}$ in $\ZZ$, so the $(e,v)$-entry of $M^*=(m^*_{e,v})$ is
$\sum_{h \in H} \mu_{v,e,h^{-1}}T_h$, and then
$$
\begin{aligned}
M(eT_{h_1}) &= \sum_{v \in V} \sum_{ h \in H} \mu_{v,e,h} v T_{h h_1}, \\
M^*(vT_{h_2}) &= \sum_{e \in E} \sum_{ h \in H} \mu_{v,e,h^{-1}} e T_{h h_2}.
\end{aligned}
$$
Therefore
$$
\langle M(eT_{h_1}), vT_{h_2} \rangle_{(\ZZ H)^V}
   =\mu_{v,e,h_2 h_1^{-1}} =
\langle eT_{h_1}, M^*(vT_{h_2}) \rangle_{(\ZZ H)^V}. \qedhere
$$
\end{proof}

The following corollary is immediate.

\begin{corollary}
\label{regular-cover-critical-group-reformulated}
For any $H$-voltage assignment $G_\beta$ and associated regular covering
$\tilde{G} \overset{\pi}{\rightarrow} G$, one has
$$
\begin{aligned}
K(\tilde{G}) &\cong (\ZZ H)^E / \left( \im \del_{G_\beta}^* + \ker \del_{G_\beta} \right) \\
              &\cong \im \del_{G_\beta} / \im \del_{G_\beta} \del_{G_\beta}^*
\end{aligned}
$$
where here 
\begin{itemize}
\item $\im \del_{G_\beta}^* $ and $\ker \del_{G_\beta}$ are $\ZZ$-sublattices of
$(\ZZ H)^E$, 
\item 
$\im \del_{G_\beta}$ and $\im \del_{G_\beta} \del_{G_\beta}^*$ are $\ZZ$-sublattices of
$(\ZZ H)^V$.
\end{itemize}
\end{corollary}

\section{The short exact sequence for a regular covering}
\label{voltage-short-exact-sequence-section}

For an $H$-voltage assignment $G_\beta$ and associated regular covering
$\tilde{G} \overset{\pi}{\rightarrow} G$, we can now identify 
the kernel of the surjection $K(\tilde{G}) \overset{\pi}{\twoheadrightarrow} K(G)$
in Proposition~\ref{covering-splitting-proposition}.

\subsection{The reduced group algebra}

\begin{definition}
Inside the group algebra $\ZZ H$, consider the (central) element 
$c:=\sum_{h \in H} T_h$, and the $2$-sided ideal $I=\ZZ c$ consisting of the
$\ZZ$-multiples of $c$.  In other words, $I$ is the $\ZZ$-submodule of $\ZZ H$
where all $\ZZ$-basis elements $T_h$ have the same coefficient.
Define the {\it reduced group algebra} $\redZH$ to be the quotient ring
$$
\redZH : = \ZZ H / I = \ZZ H / \ZZ c.
$$
Note that, just as $\ZZ H$ is a free $\ZZ$-module of rank $m:=|H|$, the ring
$\redZH$ is a free $\ZZ$-module of rank $m-1$.  As $c$ is 
invariant under $T_h \mapsto T_{h^{-1}}$, the ring
$\redZH$ inherits an anti-automorphism $\redZH \overset{*}{\longrightarrow} \redZH$
sending $\overline{T_h} \mapsto \overline{T}_{h^{-1}}$.
\end{definition}

In general we will use $\overline{(\cdot)}$ for the quotient operation 
$\ZZ H \rightarrow \redZH$ which reduces right-$\ZZ H$-modules
and morphisms modulo $I$.  For example, one has right-$\redZH$-module maps 
$$
\begin{array}{rcl}
\redZH^E &\overset{\overline{\del}_{G_\beta}}{\longrightarrow} & \redZH^V \\
\redZH^V & \overset{\overline{\del}_{G_\beta}^*}{\longrightarrow}& \redZH^E
\end{array}
$$
used in the following definition.

\begin{definition}
For $H$-voltage assignment $G_\beta$ with regular covering
$\tilde{G} \overset{\pi}{\rightarrow} G$, define the {\it critical group of $G_\beta$}
$$
\begin{aligned}
K(G_{\beta})&:= \redZH^E/  \left( \im \overline{\del}_{G_\beta}^* + \ker \overline{\del}_{G_\beta} \right)\\
&\cong 
\im \overline{\del}_{G_\beta}/
\im \overline{\del}_{G_\beta} \overline{\del}_{G_\beta}^* 
\end{aligned}
$$
where $\im \overline{\del}_{G_\beta},
\im \overline{\del}_{G_\beta} \overline{\del}_{G_\beta}^*$
are considered as $\redZH$-submodules of $\redZH^V$.
We also name the matrix in $\redZH^{V \times V}$ 
\begin{equation}
\label{voltage-graph-Laplacian}
L(G_\beta) :=\overline{\del}_{G_\beta} \overline{\del}_{G_\beta}^* 
\end{equation}
appearing in the definition of $K(G_\beta)$ the {\it voltage graph Laplacian}, so one can rewrite this as
\begin{equation}
\label{voltage-graph-critical-group-via-Laplacian}
K(G_{\beta}) \cong 
\im \overline{\del}_{G_\beta}/
\im L(G_\beta).
\end{equation}
\end{definition}

We can now prove our first main result, which was stated in the Introduction, and which we recall here.

\vskip.1in
\noindent
{\bf Theorem~\ref{regular-covering-exact-sequence}.}
{\it 
Any $H$-voltage assignment $G_\beta$ with regular covering
$\tilde{G} \overset{\pi}{\rightarrow} G$  has a short exact sequence
$$
0 \rightarrow K(G_{\beta}) \rightarrow K(\tilde{G}) \rightarrow K(G) \rightarrow 0
$$
which splits when restricted to $p$-primary components for primes $p$ not dividing $|H|$.  
In particular, $|K(\tilde{G})|=|K(G_\beta)| \cdot |K(G)|$.
}
\begin{proof}
It suffices to show that the surjection 
$K(\tilde{G}) \overset{\pi}{\twoheadrightarrow} K(G)$
from Proposition~\ref{covering-splitting-proposition} has
kernel isomorphic to $K(G_{\beta})$.  Instead we will show the equivalent
statement that $K(G_{\beta})$ is isomorphic to the cokernel of the Pontryagin dual injection 
$K(G) \overset{\pi^t}{\hookrightarrow} K(\tilde{G})$
(using Proposition~\ref{Pontryagin-duality-prop}).
This is equivalent since $\coker(\pi^t)$ is Pontryagin dual to $\ker \pi$,
and hence they are (abstractly) isomorphic abelian groups.

Recall that 
$$
\begin{aligned}
K(G) &= \ZZ^E / \left( \im \del_G^t + \ker \del_G \right), \\
K(\tilde{G}) &\cong (\ZZ H)^E / \left( \im \del_{G_\beta}^* + \ker \del_{G_\beta} \right)
\end{aligned}
$$
from Definition~\ref{edge-presentation} and Corollary~\ref{regular-cover-critical-group-reformulated}.
Consequently 
$$
\coker(\pi^t) \cong 
  (\ZZ H)^E / \left( \im \del_{G_\beta}^* + \ker \del_{G_\beta} + \pi^t(\ZZ^E) \right)
$$
Recall an edge $e$ of $G$ has fiber $\pi^{-1}(e) = \{e_h\}_{h \in H}$,
hence its basis element of $\ZZ^E$ maps under $\pi^t$ to
the sum $\sum_{h \in H} e_h$ in $\ZZ^{\tilde{E}}$.  This sum
corresponds under the isomorphism $\ZZ^{\tilde{E}} \rightarrow (\ZZ H)^E$ 
of Proposition~\ref{group-algebra-isomorphisms} to 
$$
\sum_{h \in H} e T_h =e\left( \sum_{h \in H} T_h \right)= e \cdot c.
$$ 
Hence $\pi^t(\ZZ^E) = (\ZZ c)^E=I^E$ inside $(\ZZ H)^E$, so that 
$$
(\ZZ H)^E /\pi^t(\ZZ^E) \cong (\ZZ H)^E /I^E \cong (\ZZ H /I)^E = \redZH^E.
$$
Using Noether's third isomorphism theorem,
one can conclude that
$$
\coker(\pi^t) \cong
\redZH^E / \left( \im \overline{\del}_{G_\beta}^* + \ker \overline{\del}_{G_\beta} \right)
\cong K(G_\beta),
$$
after verifying\footnote{The authors thank Julie Yuan for pointing out  (Dec. 2019) the omission of these verifications.} that the surjection $(\ZZ H)^E \rightarrow \redZH^E$
has these two properties: 
\begin{itemize}
\item[(i)] it maps $\im \del_{G_\beta}^*$ onto $ \im \overline{\del}_{G_\beta}^*$, and 
\item[(ii)] it maps $\ker \del_{G_\beta}$ onto $ \ker \overline{\del}_{G_\beta}$.  
\end{itemize}
Property (i) follows from the commutativity of this diagram with surjective horizontal maps:
\begin{equation*}
\begin{CD}
   (\ZZ H)^E  @>{}>>  \redZH^E   \\
    @V{\del_{G_\beta}^*}VV        @VV{ \overline{\del}_{G_\beta}^*}V     \\
   (\ZZ H)^V   @>{}>>    \redZH^V  
\end{CD}
\end{equation*}
Property (ii) will follow via a chase through this commutative diagram with exact rows:
\begin{equation}
\label{chased-diagram}
\begin{CD}
   0  @>{}>> \ZZ^E  @>{\pi^t}>>   (\ZZ H)^E  @>{}>>\redZH^E @>{}>>  0 \\
   @.   @V{\del_G}VV   @V{\del_{G_\beta}}VV  @VV{ \overline{\del}_{G_\beta}}V  @.    \\
   0  @>{}>> \ZZ^V @>{\pi^t}>>  (\ZZ H)^V  @>{}>>  \redZH^V  @>{}>>  0 
\end{CD}
\end{equation}
Commutativity of the right square in \eqref{chased-diagram}
shows that   $(\ZZ H)^E \rightarrow \redZH^E$ sends $\ker \del_{G_\beta}$ {\it into} 
$ \ker \overline{\del}_{G_\beta}$.  Conversely, given $\overline{\alpha}$ in $ \ker \overline{\del}_{G_\beta} \subset \redZH^E$,
we must exhibit it as the image under  $(\ZZ H)^E \rightarrow \redZH^E$ of an element of  $\ker \del_{G_\beta}$.  Start by
picking any lift $\alpha$ of $\overline{\alpha}$ in $(\ZZ H)^E$, and let $\gamma:=\del_{G_\beta}(\alpha)$ in  $(\ZZ H)^V$.  Commutativity of the 
right square  in \eqref{chased-diagram} shows that  $\gamma$ maps to $0$ under $(\ZZ H)^V \rightarrow \redZH^V$, and hence 
$\gamma=\pi^t(\gamma')$ for some $\gamma'$ in $\ZZ^V$.  We claim that $\gamma'$ lies in the image of 
$\ZZ^E \overset{\del_G}{\longrightarrow} \ZZ^V$, or equivalently,
the entries of $\gamma'$ sum to zero on the vertices
 within each connected component of $G$.  This is because $\gamma$ is in the image of $\del_{G_\beta}$, 
 so the entries of $\gamma$ sum to zero on the vertices within each connected component of $\tilde{G}$,
and hence the same must hold for $\pi (\gamma)= \pi \pi^t(\gamma') = |H| \cdot \gamma'$, and
therefore also for $\gamma'$.  Thus one can choose $\alpha'$ in $\ZZ^E$ with $\del_G(\alpha')=\gamma'$,
and then check that the element
$\alpha - \pi^t(\alpha')$ of $(\ZZ H)^E$ actually lies in $\ker(\del_{G_\beta})$, and
 maps to $\overline{\alpha}$ under $(\ZZ H)^E \rightarrow \redZH^E$.
\end{proof}

\begin{remark}
Although not needed later, for primes $p$ not dividing $m=|H|$, one can be more precise 
about the summand splitting off the $p$-primary component (or {\it Sylow $p$-subgroup}) $\Syl_p K(\tilde{G})$ 
of the critical group $K(\tilde{G})$, isomorphic to $\Syl_p K(G)$.
Since the group $H$ acts on the graph $\tilde{G}$ via graph automorphisms, it also acts on $\ZZ^{\tilde{E}}$,
preserving $\im \partial^t$ and $\ker \partial$, and inducing a (right-)action on the abelian group 
$K(\tilde{G})$.
Thus one can consider the subgroup of {\it $H$-invariants} within $K(\tilde{G})$:
$$
K(\tilde{G})^H:=\{x \in K(\tilde{G}): h(x) = x \text{ for all }h\text{ in }H\}.
$$

\begin{proposition}
In the setting of Theorem~\ref{regular-covering-exact-sequence}, for primes $p$ that do not divide $m=|H|$,
the map $\pi^t$ sends $\Syl_p K(G)$ isomorphically onto 
$
\Syl_p\left( K(\tilde{G})^H \right)$.

\end{proposition}
\begin{proof}
Note that the {\it orbit-sum map}
$
\ZZ^{\tilde{E}} \overset{\Omega}{\longrightarrow} \ZZ^{\tilde{E}}
$
sending $e \longmapsto \sum_{h \in H} h(e)$
has the same image, namely the $H$-invariants $(\ZZ^E)^H$, as does the map
$\ZZ^{E} \overset{\pi^t}{\longrightarrow} \ZZ^{\tilde{E}}$.
Since $\Omega$ is a sum of graph automorphisms, it induces a map $K(\tilde{G}) \overset{\Omega}{\longrightarrow} K(\tilde{G})$,
which again has the same image as $K(G) \overset{\pi^t}{\longrightarrow} K(\tilde{G})$.  
This image lies in $K(\tilde{G})^H$. 
Note that the map $\Omega$ when restricted from 
$K(\tilde{G})$ to $K(\tilde{G})^H$ will act as multiplication by $m$.
Hence for primes $p$ that do not divide $m$, it induces
an isomorphism 
$\Syl_p\left( K(\tilde{G})^H \right) \overset{\Omega}{\longrightarrow} \Syl_p\left( K(\tilde{G})^H \right)$.
Consequently one has
$$
\Syl_p \left( K(\tilde{G})^H \right)
=\Syl_p \left( \im \Omega \right)
=\Syl_p \left( \im \pi^t \right). \qedhere
$$
\end{proof}

On the other hand, the map $\pi^t$ generally {\it fails} to 
induce an isomorphism between $p$-primary components
of $K(G)$ and $K(\tilde{G})^H$ for primes $p$ dividing $m=|H|$.  This occurs already for $H=\ZZ/2\ZZ$
in the double covering $\tilde{G} \rightarrow G$ of an $n$-cycle $G$ by a $2n$-cycle
$\tilde{G}$, where one can check that $K(\tilde{G})^H = K(\tilde{G}) = \ZZ_{2n}$, while $K(G) = \ZZ_n$. More generally, if $H=\ZZ/m\ZZ$ in the $m$-covering $\tilde{G}\rightarrow G$ of an $n$-cycle by an $mn$-cycle $\tilde{G}$, then $K(\tilde{G})^H = K(\tilde{G}) = \ZZ_{mn}$, while $K(G) = \ZZ_n$.
\end{remark}

\section{Voltage groups of prime order}
\label{prime-order-section}

When the voltage group $H$ is abelian, the group algebra $\ZZ H$ is a commutative ring,
as is the quotient ring $\redZH$, and the distinctions between right and left modules
over these rings disappear, simplifying some of the considerations of Sections~\ref{regular-coverings-section} and \ref{voltage-short-exact-sequence-section}.

Things simplify even further if the group $H$ has prime order $p$,
as $H$ is cyclic, say with generator $h$:
$$
H=\{1,h,h^2,\cdots,h^{p-1}\} \cong \ZZ_p.
$$
Letting $\zeta$ denote a primitive $p^{th}$ root of unity in $\CC$,
one has a well-defined surjective ring map induced by
$$
\begin{array}{rccl}
\ZZ H \cong &\ZZ[T_h]/(T_h^p-1)& \rightarrow & \ZZ[\zeta] \\
            &T_h & \longmapsto & \zeta.
\end{array}
$$
Since $\zeta$ has minimal polynomial $1+x+x^2+ \cdots + x^{p-1}$ over $\QQ$,
the kernel of the above map is exactly $I=\ZZ(1 + T_h+ T_{h^2} + \cdots+T_{h^{p-1}}) = \ZZ c$, 
and hence it induces an isomorphism
$$
\redZH \cong \ZZ[\zeta].
$$
Consequently one can regard the matrices $\overline{\del}_{G_\beta}$ and 
$\overline{\del}_{G_\beta}^*$
as elements of $\ZZ[\zeta]^{V \times E}$ and $\ZZ[\zeta]^{E \times V}$, 
and one can present the
critical group for the voltage graph $G_\beta$ as
$$
\begin{aligned}
K(G_{\beta})&: = \ZZ[\zeta]^E/  \left( \im \overline{\del}_{G_\beta}^* + \ker \overline{\del}_{G_\beta} \right) \\
&\cong 
\im\overline{\del}_{G_\beta} /
\im \overline{\del}_{G_\beta} \overline{\del}_{G_\beta}^*
\end{aligned}
$$
where $\im\overline{\del}_{G_\beta}$ and $\im \overline{\del}_{G_\beta} \overline{\del}_{G_\beta}^*$
are $\ZZ[\zeta]$-submodules of $\ZZ[\zeta]^V$.
Note that under the isomorphism 
$\redZH \cong \ZZ[\zeta]$, the 
(anti-)automorphism $\overline{T}_h \mapsto \overline{T}_{h^{-1}}$
of $\redZH$ corresponds to complex conjugation $z \mapsto \bar{z}$.
Hence the matrix operation $M \mapsto M^*$ is now the usual
conjugate-transpose operation $M^*=\bar{M}^t$.

\begin{example}
Consider the regular cover $\tilde{G} \rightarrow G$ of Example~\ref{octahedron-example}, where $\tilde{G}$ is the graph of the octahedron, and the
transformation group $H =\{1,h,h^2\}$ has prime order $p=3$.  The map 
$T_h \mapsto \zeta=e^{\frac{2\pi i}{3}}$ identifies $\redZH \cong \ZZ[\zeta]$, and
under this identification one has
$$
\overline{\del}_{G_\beta}=
\bordermatrix{ 
  & a & b     & c        & d \cr
u &+1 &+1     &+1-\zeta &0  \cr
v &-1 &-\zeta &0        &+1-\zeta 
}
\qquad 
\overline{\del}_{G_\beta} \overline{\del}_{G_\beta}^*=
\bordermatrix{ 
  & u & v \cr
u &+5       &-1-\zeta \cr
v &-1-\zeta &+5 
}
$$
To understand $\im\overline{\del}_{G_\beta}, \im \overline{\del}_{G_\beta} \overline{\del}_{G_\beta}^*$,
one can use row and column operations invertible over the {\it principal ideal domain} 
$\ZZ[\zeta]$ to bring these two matrices to their unique Smith normal forms  {\it over
$\ZZ[\zeta]$}, namely 
$$
\left(
\begin{matrix}
+1 &0        &0 &0  \cr
 0 &+1-\zeta &0 &0
\end{matrix}
\right)
\quad
\text{ and}
\quad 
\left(
\begin{matrix}
1 &0 \cr
0 &24
\end{matrix}
\right).
$$
This shows that 
$$
\begin{array}{rlll}
\ZZ[\zeta]^V/\im\overline{\del}_{G_\beta} \overline{\del}_{G_\beta}^* 
 &\cong \ZZ[\zeta]/24\ZZ[\zeta]  
 &\cong \ZZ_{24}^2 
 &\cong \ZZ_{2^3}^2 \oplus \ZZ_3^2\\
\ZZ[\zeta]^V/\im\overline{\del}_{G_\beta} 
 &\cong \ZZ[\zeta]/(+1-\zeta)\ZZ[\zeta]  
 &\cong \ZZ_3
 &
\end{array}
$$
and therefore one must have
$$
\im\overline{\del}_{G_\beta}
/\im\overline{\del}_{G_\beta} \overline{\del}_{G_\beta}^* 
\cong 
\ZZ_{2^3}^2 \oplus \ZZ_3.
$$
An easy calculation shows that $K(G) = \ZZ_2$ (e.g. observe that $|K(G)|=2$ since
$G$ has only two spanning trees).
Hence the exact sequence from Theorem~\ref{regular-covering-exact-sequence}
must look as follows:
$$
\begin{array}{rcccccl}
0 \rightarrow & K(G) & \rightarrow & K(\tilde{G}) & \rightarrow & K(G_\beta) & \rightarrow 0 \\
              & \Vert&             &         &             & \Vert  & \\
              &\ZZ_2 &             &         &             & \ZZ_{2^3}^2 \oplus \ZZ_3 &
\end{array}
$$
Since the theorem tells us that this sequence splits at the 
$p$-primary components for $p \neq 3$, one concludes from this that
the octahedron graph $\tilde{G}$ has 
$$
K(\tilde{G}) \quad
\cong 
\quad
\ZZ_2 \oplus \ZZ_{2^3}^2 \oplus  \ZZ_3 
\quad
\cong 
\quad
\ZZ_2 \oplus \ZZ_8 \oplus \ZZ_{24}
$$
in agreement with the known answer (see e.g. \cite[\S 9.4.2]{BMMPR}).
\end{example}

\section{Voltage groups of order 2: double covers and signed graphs}
\label{order-two-section}

The situation is particularly simple when $p=2$, that is, for {\it double coverings}. 
As mentioned in Example~\ref{double-covers-are-regular-example},
graph double covers are {\it always} regular, with transformation group $H=\ZZ_2 =\{1,h\}$
identified with the two voltages $\{+,-\}$.  Thus 
a voltage graph $G_\beta$ as a function $E \rightarrow H=\ZZ_2=\{+,-\}$
is the same as a signed graph $G_\pm$ as defined in the Introduction.

Note that here $\zeta=-1$ and the isomorphism $\redZH \cong \ZZ[\zeta]=\ZZ$
sends $h \mapsto \zeta=-1$.  Also the anti-automorphism $T_h \mapsto T_{h^{-1}}$
of $\ZZ H$ and of $\redZH$ has become trivial, so that $\del_{G_\beta}^*=\del_{G_\beta}^t$.

\begin{proposition}
For a double cover corresponding to a signed graph $G_\pm$,
the matrix $\del_{G_\pm}$ from the Introduction 
is the same as $\del_{G_\beta}$ as in Definition~\ref{voltage-graph-Laplacian}.
In particular, the critical group $K(G_\beta)$ is exactly $K(G_\pm)$ as defined
in \eqref{signed-K(G)-presentation} in the Introduction.
\end{proposition}
\begin{proof}
In the Introduction, $\del_{G_\pm}$ mapped a positive (resp. negative) edge $e$ directed 
as $(u,v)$ to $+u-v$ (resp. $+u+v$),
which agrees with the action of $\del_{G_\beta}$ as $e \longmapsto +u-v T_{\beta(e)}$,
since $T(\beta(e)) \mapsto +1,-1$ depending upon whether $e$ is a positive, negative edge.
\end{proof}

\noindent
Consequently, Theorem~\ref{regular-covering-exact-sequence} immediately implies the
following result from the Introduction.

\vskip.1in
\noindent
{\bf Theorem~\ref{unsigned-double-cover-theorem}.} 
{\it 
For each signed graph $G_{\pm}$, parametrizing a graph double covering
$\tilde{G} \rightarrow G$, one has a short exact sequence of critical groups
$$
0 \rightarrow K(G_\pm) \rightarrow K(\tilde{G}) \rightarrow K(G) \rightarrow 0
$$
splitting on restriction to $p$-primary components for odd primes $p$.
In particular, $|K(\tilde{G})| = |K(G_\pm)| \cdot |K(G)|$.
}
\vskip.1in
\noindent

\noindent
Theorem~\ref{unsigned-double-cover-theorem} will be generalized in a different 
direction in Theorem~\ref{signed-graph-double-cover-complex} below, 
after we generalize (in Section~\ref{signed-graph-double-cover-section})
the notion of 
double coverings of unsigned graphs to {\it double coverings of signed graphs}.

\subsection{Example: Bipartite double covers and crowns}

\begin{definition}
Given an unsigned multigraph $G=(V,E)$, its {\it bipartite double cover} 
(see, e.g. Waller \cite{Waller}) is the double cover $\tilde{G} \rightarrow G$
associated to the signed graph which 
Zaslavsky \cite[\S 7.D]{Zaslavsky} calls the {\it all-negative assignment} $G_\beta=G_\pm=-G$, in which every edge $e$ in $E$ has $\beta(e)=-$.
The bipartite double cover of $G$ is
sometimes also called the {\it tensor product} or {\it categorical product}
$G \times K_2$, where $K_2$ is the unsigned graph consisting of a single edge between
two vertices.
\end{definition}

When $G$ is highly symmetric, the same is true of the all negative signed graph $-G$, sometimes leading to an easy computation of both $K(G), K(-G)$,
where Theorem~\ref{unsigned-double-cover-theorem} is easy to apply.

\begin{example}
\label{crown-example}
The {\it $n$-crown graph} $\crown_n$ is the unsigned graph obtained from the 
complete bipartite graph $K_{n,n}$ on bipartitioned
vertex set $V=\{v^{(1)}_+,\ldots,v^{(n)}_+\}  \sqcup \{v^{(1)}_-,\ldots,v^{(n)}_-\}$ by removing the
perfect matching of edges $M=\{ \{v^{(i)}_+,v^{(i)}_-\}:i=1,2,\ldots,n \}$.  
More generally, define $\crown_n^{(k)}$ to be the multigraph obtained from $\crown_n$ by adding back in $k$ copies of each edge from the 
perfect matching $M$ that was removed.  Equivalently, 
$\crown_n^{(k)}$ is the multigraph obtained from $K_{n,n}$ by
adding $k-1$ copies of the perfect matching $M$.
In particular, taking $k=1$, the
graph $\crown_n^{(1)}$ recovers $K_{n,n}$ itself.

Let $K_n^{(m)}$ be the multigraph obtained from the complete graph $K_n$ on vertex set
$\{v^{(1)},\ldots,v^{(n)}\}$ by adding $m$ multiple copies of a self-loop to every vertex $v^{(i)}$. 

The following proposition is then straightforward.

\begin{proposition}
{\it When $k$ is even}, $\crown_n^{(k)}$ provides
the bipartite double covering of $K_n^{(\frac{k}{2})}$,
via the map 
$$
\begin{array}{rcl}
\tilde{G}:=\crown_n^{(k)} & \overset{\pi}{\longrightarrow} 
                              & K_n^{(\frac{k}{2})}=:G\\
v^{(i)}_+& \longmapsto & v^{(i)} \\
v^{(i)}_- & \longmapsto & v^{(i)} \\
\{v^{(i)}_+,v^{(j)}_-\}  & \longmapsto & \{v^{(i)},v^{(j)}\}\\
\{v^{(j)}_+,v^{(i)}_-\}  & \longmapsto & \{v^{(i)},v^{(j)}\}
\end{array}
$$
that also sends the extra $k$ copies of the
matching edge $\{v^{(i)}_+,v^{(i)}_-\}$ to the $\frac{k}{2}$ copies of
the loop edge on $v^{(i)}$.  
\end{proposition}

\begin{example}
For $n=4, k=2$, here is a depiction of 
the bipartite double covering
$\crown_4^{(2)} \rightarrow K_4^{(1)}$:

\begin{center}
\begin{tikzpicture}[scale=0.75, auto,swap,every loop/.style={}]
    \foreach \pos/\name in {{(0,5)/v^{(1)}_+},{(3,4)/v^{(2)}_+},
                               {(6,4)/v^{(3)}_+},{(9,5)/v^{(4)}_+}}
        \node[vertex] (\name) at \pos {$\name$};
   \foreach \pos/\name in {{(0,1)/v^{(1)}_-},{(3,0)/v^{(2)}_-},
                               {(6,0)/v^{(3)}_-},{(9,1)/v^{(4)}_-}}
        \node[vertex] (\name) at \pos {$\name$};
 
    \foreach \source/ \dest in {
        v^{(1)}_+/v^{(2)}_-,
        v^{(1)}_+/v^{(3)}_-,
        v^{(1)}_+/v^{(4)}_-,
        v^{(2)}_+/v^{(3)}_-,
        v^{(2)}_+/v^{(4)}_-,
        v^{(2)}_+/v^{(1)}_-,
        v^{(3)}_+/v^{(2)}_-,
        v^{(3)}_+/v^{(4)}_-,
        v^{(3)}_+/v^{(1)}_-,
        v^{(4)}_+/v^{(1)}_-,
        v^{(4)}_+/v^{(2)}_-,
        v^{(4)}_+/v^{(3)}_-}
        \path[edge] (\source) -- (\dest);

    \foreach \source/ \dest in {v^{(1)}_+/v^{(1)}_-,v^{(2)}_+/v^{(2)}_-,v^{(3)}_+/v^{(3)}_-,v^{(4)}_+/v^{(4)}_-}
       \path (\source) edge [bend left] (\dest);
    \foreach \source/ \dest in {v^{(1)}_+/v^{(1)}_-,v^{(2)}_+/v^{(2)}_-,v^{(3)}_+/v^{(3)}_-,v^{(4)}_+/v^{(4)}_-}
       \path (\source) edge [bend right] (\dest);

\draw[arrows=->,line width=1pt](4.5,0)--(4.5,-1);

   \foreach \pos/\name in {{(0,-1.5)/v^{(1)}},{(3,-2.5)/v^{(2)}},
                               {(6,-2.5)/v^{(3)}},{(9,-1.5)/v^{(4)}}}
        \node[vertex] (\name) at \pos {$\name$};

    \foreach \source/ \dest in {
        v^{(1)}/v^{(2)},
        v^{(1)}/v^{(3)},
        v^{(1)}/v^{(4)},
        v^{(2)}/v^{(3)},
        v^{(2)}/v^{(4)},
        v^{(3)}/v^{(4)}}
        \path[edge] (\source) -- (\dest);

\foreach \source in {v^{(1)},v^{(2)},v^{(3)},v^{(4)}}    
	\path[every node/.style={font=\sffamily\small}]
        (\source)   edge[in=-50,out=-130,loop]  (\source);

\end{tikzpicture}
\end{center}
\end{example}

\begin{corollary}
\label{first-crown-corollary}
For $k$ even\footnote{We will be able to remove this assumption that $k$ is even in Section~\ref{revisited-crown-section} below, 
after allowing for negative half-loops in signed graphs and double covers.} 
and $n$ odd, 
$$
K(\crown_n^{(k)}) 
\cong \ZZ_n^{n-2} \oplus \ZZ_{n-2+2k}^{n-2} \oplus \ZZ_{(n-1+k)(n-2+2k)}^{n-2} 
$$
\end{corollary}
\begin{proof}
For $G=K_n^{(\frac{k}{2})}$, both the unsigned graph Laplacian $L(G)=\del_G \del_G^t$
and the all-negative signed graph Laplacian $L(-G)=\del_{-G} \del_{-G}^t$
are $n \times n$ matrices of the form 
$
M_n(b,a) = b I_{n \times n} - a J_{n \times n}
$
where $I, J$ are the identity and all ones matrices, respectively.
Specifically, 
$$
\begin{aligned}
L(G)&=M_n(n,1) \\
L(-G)&=M_n(n-2+2k,-1). \\
\end{aligned}
$$
Hence one can begin the calculation of $K(G)$ and $K(-G)$ with an
easy general computation (see \cite[Prop 4.2(v)]{Jacobson})
showing $M_n(b,a)$ has Smith normal form entries
$$
\left(\gcd(a,b), \,\, \underbrace{b,b,\ldots,b}_{n-2\text{ times}}, \,\,  
\frac{b(b-na)}{\gcd(a,b)} \right).
$$
For $L(G)$ this gives Smith normal form entries
$
(1,n,\ldots,n, 0 )
$
and $K(G)=\im \del_G / \im L(G) \cong \ZZ_n^{n-2}$, as is well-known.
For $L(-G)$ it gives Smith entries
$$
(1, \,\, \underbrace{n-2+2k,\ldots,n-2+2k}_{n-2\text{ times}},  \,\,2(n-1+k)(n-2+2k) )
$$
and hence 
\begin{equation}
\label{minus-complete-smith-quotient}
\ZZ^n/ \im L(-G) 
\quad \cong \quad
\ZZ_{n-2+2k}^{n-2} \oplus \ZZ_{2(n-1+k)(n-2+2k)}.
\end{equation}
One can also easily check 
(see Proposition~\ref{connected-signed-graph-im-of-del} below) 
that $\im \del_{-G}$ is the index two sublattice $\ZZ^n_{\equiv 0\bmod{2}}$
 of $\ZZ^n$ where the sum of the entries is even.  
Hence $K(-G)=\im \del_{-G} / \im L(-G)$
must be a subgroup of index two within the group $\ZZ^n/ \im L(-G)$
described in \eqref{minus-complete-smith-quotient} above.
If {\it one assumes that $n$ is odd}, which we will do for the remainder of this calculation, so that $n-2+2k$ is also odd, then the only summand
in  \eqref{minus-complete-smith-quotient} having a subgroup of
index $2$ is the last summand $\ZZ_{2(n-1+k)(n-2+2k)}$.
Hence this forces
$$
K(-G) \quad \cong \quad \ZZ_{n-2+2k}^{n-2} \oplus \ZZ_{(n-1+k)(n-2+2k)}
$$
for $n$ odd\footnote{Actually, with a bit more matrix manipulation, 
one can draw this same conclusion {\it for all $n$}; see Tseng \cite[\S 8.1]{Tseng}.}.
Thus the short exact sequence from 
Theorem~\ref{unsigned-double-cover-theorem} 
takes the form

\begin{equation}
\label{crown-short-exact-sequence}
\begin{array}{rcccccl}
0 \rightarrow & K(G) & \rightarrow & K(\tilde{G}) & \rightarrow & K(-G) & \rightarrow 0 \\
              & \Vert&             &         &             & \Vert  & \\
              & &             &            &             &    & \\
              &\ZZ_n^{n-2} &             &            &             & \ZZ_{n-2+2k}^{n-2}  & \\
              & &             &            &             &  \oplus  & \\
              & &             &            &             &  \ZZ_{(n-1+k)(n-2+2k)}. &
\end{array}
\end{equation}
Since the theorem also tells us this sequence splits at 
$p$-primary components for all odd primes $p$, and since $K(G)=\ZZ_n^{n-2}$ 
only has odd primary components for $n$ odd, the sequence must
split at all primes.  Therefore 
$$
K(\crown_n^{(k)}) 
\quad = \quad
K(\tilde{G}) 
\quad \cong \quad K(G) \oplus K(-G) 
\quad \cong \quad
\ZZ_n^{n-2} \oplus \ZZ_{n-2+2k}^{n-2} \oplus \ZZ_{(n-1+k)(n-2+2k)}^{n-2}. 
\qedhere
$$
\end{proof}

\noindent
We remark that, for $k=0$, this answer for $n$ odd agrees with
a result of Machacek \cite[Theorem 14]{Machacek} 
\begin{equation}
\label{Machacek's-crown-calculation}
K(\crown_n) \quad \cong \quad 
\ZZ_{n-2}
\oplus
\ZZ^{n-3}_{n(n-2)}
\oplus
\ZZ_{n(n-1)(n-2)}
\end{equation}
proven correct for {\it all} $n$ (not just $n$ odd)
via Smith normal forms.  
See also Remark~\ref{Machacek-n-even-crown-remark}
below.
\end{example}

\section{Application: when the voltage graph Laplacian is diagonal}
\label{diagonal-application-section}

The voltage graph Laplacian $L(G_\beta)$ defined in
\eqref{voltage-graph-Laplacian} has a peculiar feature
that happens only when the voltage group $H$ is nontrivial:  
nonempty voltage graphs $G_\beta$ can have a {\it diagonal} $L(G_\beta)$.
We describe such a situation, giving a result that
uses this diagonal structure, then apply it to three families of examples.

\subsection{The construction}

\begin{definition}
For a positive integer $m \geq 2$ and a multigraph $G=(V,E)$, let
$mG=(V,mE)$ denote the multigraph on the same vertex set $V$ 
in which each edge $e$ in $E$ has been replicated into $m$ copies.

Given a group $H$ of order $|H|=m$, and a multigraph $G=(V,E)$,
let $HG$ denote the $H$-voltage graph whose underlying multigraph is $mG$, so that
its edges can be labelled $\{e^{(h)}\}_{e \in E, h \in H}$, and
with voltage assignment $\beta(e^{(h)})=h$.
\end{definition}

\begin{proposition}
\label{diagonal-voltage-Laplacian-prop}
Consider  a group $H$ of order $m \geq 2$, and
a connected multigraph $G=(V,E)$ with degree sequence $(d_1,\ldots,d_{|V|})$
of $G$, in which loops count $2$ toward
the degree of a vertex.
Then the voltage graph Laplacian $L(HG)$ in $\redZH^{V \times V}$
is the diagonal matrix whose entries are $(md_1,\ldots,md_{|V|})$.

Furthermore, after uniquely expressing
$
\bigoplus_{i=1}^{|V|} \ZZ_{d_i} \cong \bigoplus_{i=1}^{|V|} \ZZ_{s_i}
$
for positive integers $s_1,\ldots,s_{|V|}$ with $s_i$ dividing $s_{i+1}$, one has
\begin{equation}
\label{diagonal-critical-group}
K(HG) \cong \ZZ_{s_1} \oplus \ZZ_{ms_1}^{m-2} \oplus
\bigoplus_{i=2}^{|V|} \ZZ_{ms_i}^{m-1}.
\end{equation}
In particular, whenever $m$ is relatively prime to all the
degrees $d_i$, one can rewrite this as
$$
K(HG) \cong \ZZ_m^{(m-1)|V|-1} \oplus \bigoplus_{i=1}^{|V|} \ZZ_{d_i}^{m-1}.
$$
\end{proposition}

\begin{proof}
For the description of the entries of $L(HG)$, first note that 
$L(HG)$ is diagonal since a 
pair of vertices $u,v$ with $u \neq v$ having $d$ edges 
between them will have $(u,v)$ entry in $L(HG)$ given by 
$$
d \sum_{h \in H} (+1)(-\overline{T}_h ) = -d \sum_{h \in H} \overline{T}_h = -d \cdot \overline{c}= 0 \quad \text{ in } \redZH(:=\ZZ H /\ZZ c).
$$
Thus we only need to compute the diagonal $(v,v)$ entry corresponding to each 
vertex $v$ in $V$.
If $v$ has $\ell$ loops attached and is incident to $d$ nonloop edges,
then this $(v,v)$ entry in
$L(HG):=\overline{\del}_{G_\beta} \overline{\del}_{G_\beta}^*$ is given by the sum 
$$
\begin{aligned}
&d \sum_{h \in H} (-\overline{T}_h) (-\overline{T}_{h^{-1}}) +
\ell \sum_{h \in H} (1-\overline{T}_h)(1-\overline{T}_{h^{-1}})\\
&=d \sum_{h \in H} 1 +
\ell \sum_{h \in H} \left( 2-(\overline{T}_h+\overline{T}_{h^{-1}}) \right)\\
 &= dm+2\ell m\\
&= m d_v.
\end{aligned}
$$

For the assertions about $K(HG)$, we use its presentation from
\eqref{voltage-graph-critical-group-via-Laplacian}
as 
$
K(HG)=\im \overline{\del}_{HG} /\im L(HG),
$
and start by describing $\im \overline{\del}_{HG}$
more explicitly. 
Note that $\overline{\del}_{HG}$ fits into this
commutative square, where the vertical maps are both quotient maps:
$$
\begin{CD}
(\ZZ H)^{E} @>\del_{H G}>> (\ZZ H)^V \\
@V\kappa_{E}VV  @VV\kappa_{V}V \\
(\overline{\ZZ H})^{E} @>\overline{\del}_{H G}>> (\overline{\ZZ H})^V. \\
\end{CD}
$$
Therefore  
$
\im \overline{\del}_{HG}=\kappa_V(\im \del_{HG}),
$
and it helps to first analyze $\im \del_{HG}$.
If $\widetilde{mG}=(\widetilde{V},\widetilde{E})$ denotes
the total space in the covering $\widetilde{mG}\rightarrow mG$,
then one easily checks (or see Proposition~\ref{mGstructure} below)
that connectivity of $G$ implies connectivity of 
$\widetilde{mG}$.  Hence $\im \del_{H G}$ is the sublattice 
$\ZZ^{\widetilde{V}}_{=0}$ 
of $\ZZ^{\widetilde{V}}$ 
where the coordinates sum to zero. 
Under the isomorphism of $\ZZ^{\widetilde{V}}$ with 
$(\ZZ H)^{V}$ in Proposition \ref{group-algebra-covering-identifications}, 
this sublattice $\im\del_{H G}$ corresponds to
the sublattice of $(\ZZ H)^{V}$ consisting of those 
elements $x=(x_1,\ldots,x_{|V|})$ whose sum of coordinates 
$x_1 + \cdots + x_{|V|}=\sum_{h \in H} a_h T_h$,
when considered as an element of $\ZZ H$, satisfies $\sum_{h \in H} a_h=0$.
Then $\im \overline{\del}_{HG}$ is the image
of this sublattice $\im\del_{H G}$ of $(\ZZ H)^{V}$
under the quotient map $\kappa_{V}$ 
that mods out by multiples of $c:=\sum_{h \in H} T_h$.
Since $c$ has its sum of coordinates equal to $|H|=m$,
one concludes that $\im \overline{\del}_{HG}$
is the sublattice $\Lambda$ of $(\overline{\ZZ H})^V$ consisting
of the elements $x=(x_1,\ldots,x_{|V|})$ whose sum of coordinates 
$x_1 + \cdots + x_{|V|}=\sum_{h \in H} a_h \overline{T}_h$,
when considered as an element of $\overline{\ZZ H}$, satisfies
$\sum_{h \in H} a_h \equiv 0 \bmod{m}$.


We next compute that
\begin{equation}
\label{Lambda-mod-diagonal-expression}
\begin{aligned}
K(HG) &= \im \overline{\del}_{HG}/\im L(HG)  \\
&= \Lambda \left/ \bigoplus_{i=1}^{|V|} md_i \redZH \right.\\
&= \ZZ^{(m-1)|V|}_{\equiv 0 \bmod{m}} \left/ \bigoplus_{i=1}^{|V|} md_i \ZZ^{m-1} \right.
\end{aligned}
\end{equation}
in which $\ZZ^n_{\equiv 0 \bmod{m}}$ denotes the sublattice of 
$\ZZ^n$ where the sum of
coordinates is $0$ modulo $m$.  Then the last expression in 
\eqref{Lambda-mod-diagonal-expression}
is isomorphic to the right-side of \eqref{diagonal-critical-group}
via Lemma~\ref{numerical-lemma} below.

For the last assertion of the proposition, when $m$ happens to be relatively prime to all the
vertex degrees $d_i$, it is also relatively prime to all of the $s_i$,
and hence one has
$$
\begin{aligned}
K(HG) 
&\cong \ZZ_{s_1} \oplus \ZZ_{ms_1}^{m-2} \oplus
\bigoplus_{i=2}^{|V|} \ZZ_{ms_i}^{m-1} \\
&\cong \ZZ_m^{(m-1)|V|-1} \oplus \bigoplus_{i=1}^{|V|} \ZZ_{s_i}^{m-1} \\
&\cong \ZZ_m^{(m-1)|V|-1} \oplus \bigoplus_{i=1}^{|V|} \ZZ_{d_i}^{m-1}. 
\qedhere
\end{aligned}
$$
\end{proof}

The following numerical lemma was used in the preceding proof.
\begin{lemma}
\label{numerical-lemma}
Given positive integers $d_1,\ldots,d_n$, if one uniquely expresses 
$
\bigoplus_{i=1}^{n} \ZZ_{d_i} \cong \bigoplus_{i=1}^{n} \ZZ_{s_i}
$
for positive integers $s_1,\ldots,s_{|V|}$ with $s_i$ dividing $s_{i+1}$, then
$$
\ZZ^n_{\equiv 0 \bmod {m}} \left/ \bigoplus_{i=1}^n md_i \ZZ \right.
\quad 
\cong 
\quad
\ZZ_{s_1} \oplus \bigoplus_{i=2}^n \ZZ_{ms_i} .
$$
\end{lemma}
\begin{proof}
If $\ZZ^n$ has standard basis $\epsilon_1,\ldots,\epsilon_n$,
then the sublattice $\ZZ^n_{\equiv 0 \bmod {m}}$ has a $\ZZ$-basis
given by $\delta_1=m \epsilon_1$ and $\delta_i=\epsilon_i - \epsilon_{i-1}$ for
$i=2,3,\ldots,n$.  With respect to this basis, one can express
$
m \epsilon_i = \delta_1 + m(\delta_2+\delta_3+\cdots+\delta_i)
$
for $i=1,2,\ldots,n$.  Hence 
$
\ZZ^n_{\equiv 0 \bmod {m}} \left/ \bigoplus_{i=1}^n md_i \ZZ \right.
$
is isomorphic to the cokernel of this matrix $A$ in $\ZZ^{n \times n}$:
$$
A=\left[
\begin{matrix}
d_1   & d_2   & d_3 & \cdots & d_n \\
0     &md_2   &md_3 & \cdots &md_n \\
0     &0      &md_3 & \cdots &md_n \\
\vdots&\vdots &     &        &\vdots\\
0     &0      &0    &  \cdots&md_n\\
\end{matrix}
\right]=
\left[
\begin{matrix}
1     & 0   & 0 & \cdots & 0  \\
0     &m   & 0 & \cdots &0  \\
0     &0      &m & \cdots &0  \\
\vdots&\vdots &     &        &\vdots\\
0     &0      &0    &  \cdots&m\\
\end{matrix}
\right]
\left[
\begin{matrix}
d_1   & d_2   & d_3 & \cdots & d_n \\
0     &d_2   &d_3 & \cdots &d_n \\
0     &0      &d_3 & \cdots &d_n \\
\vdots&\vdots &     &        &\vdots\\
0     &0      &0    &  \cdots&d_n\\
\end{matrix}
\right].
$$
One can easily check that for $k=1,2,\ldots,n$,
the gcd of the set of all $k \times k$ minor subdeterminants
of $A$ is $m^{k-1}$ times the gcd of all products $d_{i_1} \cdots d_{i_k}$
with $1 \leq i_1 < ... < i_k \leq n$, and hence equals $m^{k-1} s_1 s_2 \cdots s_k$.
This implies the Smith normal form entries for $A$ are
$s_1, ms_2, ms_3,\ldots, ms_n$, proving the lemma.
\end{proof}

Having determined the critical group $K(HG)$ for 
these special voltage graphs $HG$,
we now wish to consider their associated regular covering of
the underlying graph $mG$.  

\begin{proposition}
\label{mGstructure}
For $m \geq 2$, the total space $\widetilde{mG}=(\tilde{V},\tilde{E})$ 
in the graph covering 
$\widetilde{mG} \rightarrow mG$ associated to $HG$ has
the following description as an undirected graph, which depends
only on $m=|H|$, and not on the structure of $H$ as a group.

The vertex set $\tilde{V}$ contains $m$ 
vertices $\{v_h\}_{h \in H}$ for each vertex $v$ in $V$.

The edge set $\tilde{E}$ contains for each
nonloop edge $e=(u,v)$ of $E$, a copy of the
complete bipartite graph $K_{m,m}$ on 
the bipartitioned vertex set 
   $
    \{u_h\}_{h \in H} \sqcup \{v_h\}_{h \in H}.
   $
For each loop edge on a vertex $v$ in $V$,
the edge set $\tilde{E}$ also contains
a loop on each vertex in $\{v_h\}_{h \in H}$, as well as
two copies
of the complete graph $K_m$ on vertex set $\{v_h\}_{h \in H}$.
\end{proposition}
\begin{proof}
An edge $(u_{h_1},v_{h h_1})$ within a copy of the bipartite graph
$K_{m,m}$ corresponding to an edge $e$ represents 
in the regular cover the edge labelled $e_{h_1}^{(h)}$, lying above the copy   
$e^{(h)}$ of $e$ in $mG$ that has been assigned voltage $\beta(e^{(h)})=h$.

For each loop $e$ at a vertex $v$ in $V$,
the loop on $v_h$ represents the regular cover edge
labelled $e^{(\one)}_h$ lying above the copy $e^{(\one)}$ of $e$ in
$mG$ that has been assigned voltage $\beta(e^{(\one)})=\one$.
The two copies of the complete graphs $K_m$ on $\{v_h\}_{h \in H}$
come from the regular cover edges  labelled 
$e^{(h)}_{h'}$ as the ordered pairs $(h,h')$ run through
$H \times (H \setminus \{\one\})$:  for $h_1 \neq h_2$ in $H$ one will
obtain the undirected edge $\{v_{h_1},v_{h_2}\}$ twice,
once from taking $h=h_1$ and $h'=h_2 h_1^{-1}$, and once
from taking $h=h_2$ and $h'=h_1 h_2^{-1}$.
\end{proof}

\begin{corollary}
\label{stacked-bipartite-corollary}
Fix $m \geq 2$ a positive integer, and let 
$G=(V,E)$ be a connected multigraph with no loops, with critical group 
$
K(G) \cong \bigoplus_{i=1}^{|V|-1} \ZZ_{k_i}.
$
Assume $G$ has vertex degrees 
$d_1,\ldots,d_{|V|}$, all relatively prime to $m$.

Then the short exact sequence of Theorem~\ref{regular-covering-exact-sequence} becomes
\begin{equation}
\label{diagonal-Laplacian-short-exact}
\begin{array}{rcccccccl}
0 &\rightarrow & K(mG) & \rightarrow & K(\widetilde{mG}) & \rightarrow & K(HG) & \rightarrow &0. \\
&              & \Vert&             &              &             & \Vert  & & \\
              &             &  & &\\
 &             &\bigoplus_{i=1}^{|V|-1} \ZZ_{m k_i}&             &             
              &             & \ZZ_m^{(m-1)|V|-1} & &\\
           & & & & &             & \oplus  & &\\
           & & & & &             & \bigoplus_{i=1}^{|V|} \ZZ_{d_i}^{m-1}  & &
\end{array}
\end{equation}
If one further assumes that $m$ is relatively prime to the determinant of the adjacency matrix of $G$, then
\begin{equation}
\label{diagonal-Laplacian-short-exact-with-further-assumption}
K(\widetilde{mG}) 
\quad \cong \quad
\ZZ_m^{(m-2)|V|} 
  \oplus \bigoplus_{i=1}^{|V|-1} \ZZ_{m^2 k_i}
  \oplus \bigoplus_{i=1}^{|V|} \ZZ_{d_i}^{m-1}.
\end{equation}
\end{corollary}

\begin{proof}
The fact that
$
K(G) \cong \bigoplus_{i=1}^{|V|-1} \ZZ_{k_i}
$
 if and only if 
$
K(mG) \cong \bigoplus_{i=1}^{|V|-1} \ZZ_{m k_i}
$
is a well-known consequence of the 
presentation \eqref{K(G)-Kirchoff-presentation} of $K(G)$,
as one has this relation between Laplacians matrices: $L(mG)=mL(G)$.
The description of $K(HG)$ comes 
from Proposition~\ref{diagonal-critical-group}.  
This explains the exact sequence \eqref{diagonal-Laplacian-short-exact}.

For the last assertion, we first consider primes $p$ that divide $m$, noting
that the $p$-primary part of the sequence \eqref{diagonal-Laplacian-short-exact} looks like
\begin{equation}
\label{p-primary-sequence}
0 \rightarrow 
   \Syl_p \left( \bigoplus_{i=1}^{|V|-1} \ZZ_{m k_i}  \right)
  \rightarrow 
    \Syl_p K(\widetilde{mG}) 
  \rightarrow 
    \Syl_p \left( \ZZ_m^{(m-1)|V|-1} \right) 
       \rightarrow 0 
\end{equation}
where recall that here $\Syl_p(A)$ denotes the {\it Sylow $p$-subgroup} or 
{\it $p$-primary component}
of a finite abelian group $A$.
Since $m \geq 2$, one can pick distinct 
elements $h_1 \neq h_2$ in $H$, and check that the
$|V| \times |V|$ submatrix of
the Laplacian $L(\widetilde{mG})$ with
rows indexed by $\{v_{h_1}\}_{v \in V}$ and columns indexed by
$\{v_{h_2}\}_{v \in V}$ is the negative of the adjacency matrix of $G$.
Therefore under the additional assumption that $m$ is relatively prime to 
the determinant of this adjacency matrix, 
for every prime $p$ dividing $m$,
the $p$-primary component $\Syl_p K(\widetilde{mG})$ appearing in the middle of
the sequence \eqref{p-primary-sequence} has its number of generators bounded by
$$
(|\tilde{V}|-1)-|V| = (m|V| - 1) -|V| = (m-1)|V|-1
$$
matching the exponent on $\ZZ_m$ in the 
right term in the sequence.  This then forces 
$$
\Syl_p K(\widetilde{mG}) 
=  \Syl_p \left( \bigoplus_{i=1}^{|V|-1} \ZZ_{m^2 k_i}
\oplus \ZZ_m^{(m-2)|V|} \right)
$$
By Proposition \ref{covering-splitting-proposition}, the sequence 
\eqref{diagonal-Laplacian-short-exact}
splits at $p$-primary components for
all the other primes $p$ that do {\it not} divide $m$.
Collating the various $p$-primary components then gives
the description \eqref{diagonal-Laplacian-short-exact-with-further-assumption} 
for $K(\widetilde{mG})$.
\end{proof}

We next apply Corollary~\ref{stacked-bipartite-corollary} in three 
families of examples.

\subsection{Example:  when $G$ is a path}

\begin{proposition}
\label{stacked-bipartite-paths}
Let $G=(V,E)$ be a path, with $|V|$ even, and let $m \geq 3$ be an odd integer.
Then the regular covering $\widetilde{mG} \rightarrow mG$ has
$$
K(\widetilde{mG}) =  
\ZZ_m^{(m-2)|V|}
\oplus
\ZZ_{m^2}^{|V|-1}
\oplus
\ZZ_2^{(m-1)(|V|-2)}.
$$
\end{proposition}
Here is a picture of $\widetilde{mG} \rightarrow mG$ 
in the case where $m=3$ and $|V|=4$. 

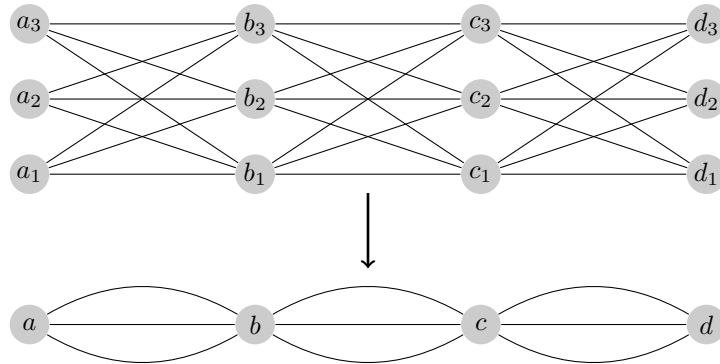
\begin{center}
\begin{tikzpicture}[scale=1, auto,swap,every loop/.style={}]
    \foreach \pos/\name in {{(0,1)/a_1},{(3,1)/b_1},{(6,1)/c_1},{(9,1)/d_1}}
        \node[vertex] (\name) at \pos {$\name$};
    \foreach \pos/\name in {{(0,2)/a_2},{(3,2)/b_2},{(6,2)/c_2},{(9,2)/d_2}}
        \node[vertex] (\name) at \pos {$\name$};
    \foreach \pos/\name in {{(0,3)/a_3},{(3,3)/b_3},{(6,3)/c_3},{(9,3)/d_3}}
        \node[vertex] (\name) at \pos {$\name$};
    \foreach \source/ \dest in {a_1/b_1,a_1/b_2,a_1/b_3,a_2/b_1,a_2/b_2,a_2/b_3,a_3/b_1,a_3/b_2,a_3/b_3,b_1/c_1,b_1/c_2,b_1/c_3,b_2/c_1,b_2/c_2,b_2/c_3,b_3/c_1,b_3/c_2,b_3/c_3,c_1/d_1,c_1/d_2,c_1/d_3,c_2/d_1,c_2/d_2,c_2/d_3,c_3/d_1,c_3/d_2,c_3/d_3}
        \path[edge] (\source) -- (\dest);

\draw[arrows=->,line width=1pt](4.5,.75)--(4.5,-.25);

    \foreach \pos/\name in {{(0,-1)/a},{(3,-1)/b},{(6,-1)/c},{(9,-1)/d}}
        \node[vertex] (\name) at \pos {$\name$};
    \foreach \source/ \dest in {a/b,b/c,c/d}
        \path[edge] (\source) -- (\dest);

    \foreach \source/ \dest in {a/b,b/c,c/d}
       \path (\source) edge [bend left] (\dest);

    \foreach \source/ \dest in {a/b,b/c,c/d}
       \path (\source) edge [bend right] (\dest);
\end{tikzpicture}
\end{center}

\begin{proof}
Check Corollary~\ref{stacked-bipartite-corollary} applies. 
The vertex degrees are $(d_1,\ldots,d_{|V|})=(1,2,2,\ldots,2,2,1)$,
all relatively prime to the odd number $m$.  
One can calculate by induction on $|V|$
that a path has the determinant of its adjacency matrix either
$0$ for $|V|$ odd, or $\pm 1$ for $|V|$ even, and hence
relatively prime to $m$ when $|V|$ is even.  
Having checked
that the corollary applies, one needs to know that
$K(G)$ is the trivial group (since $G$ is a tree), with
invariant factors $(k_1,\ldots,k_{|V|-1})=(1,1,\ldots,1)$.
\end{proof}

\subsection{Example:  when $G$ is a cycle}

\begin{proposition}
\label{stacked-bipartite-cycles}
Let $G=(V,E)$ be a cycle with $|V| \not\equiv 0 \bmod{4}$, 
and let $m \geq 3$ be an odd integer.
Then the regular covering $\widetilde{mG} \rightarrow mG$ has
$$
K(\widetilde{mG}) =  
\ZZ_m^{(m-2)|V|}
\oplus
\ZZ_{m^2|V|}
\oplus
\ZZ_{m^2}^{|V|-2} 
\oplus
\ZZ_2^{(m-1)|V|}.
$$
\end{proposition}

Note that in the previous figure, gluing together vertices $a,d$ and gluing
$a_i, d_i$ to each other for $i=1,2,3$, gives the picture for $m=3=|V|$.
\begin{proof}
Again check Corollary~\ref{stacked-bipartite-corollary} applies.
The vertex degrees $(d_1,\ldots,d_{|V|})=(2,2,\ldots,2)$,
relatively prime to the odd number $m$.
One can calculate by induction on $|V|$
(or see \cite[Prop. 2.4]{Abdollahi})
that a cycle has the determinant of its adjacency matrix either
$$
\begin{cases}
0 &\text{ if } |V| \equiv 0 \bmod{4}, \\
-4 & \text{ if } |V| \equiv 2 \bmod{4}, \\
\pm 2 & \text{ if } |V| \equiv 1,3 \bmod{4},
\end{cases}
$$
relatively prime to the odd number $m$ 
if $|V| \not\equiv 0 \bmod{4}$.
Having checked
that the corollary applies, one needs the
well-known fact that $K(G)=\ZZ_{|V|}$, with invariant factors 
$(k_1,\ldots,k_{|V|-1})=(1,1,\ldots,1,|V|)$.
\end{proof}

\subsection{Example:  when $G$ is a complete graph}

\begin{proposition}
\label{stacked-bipartite-completes}
Let $G=(V,E)$ be a complete graph (without loops) having $|V|-1$ relatively prime to the positive number
$m \geq 2$.  Then the regular covering $\widetilde{mG} \rightarrow mG$ has
$$
K(\widetilde{mG}) =  
\ZZ_m^{(m-2)|V|}
\oplus
\ZZ_{m^2}
\oplus
\ZZ_{m^2|V|}^{|V|-2} 
\oplus
\ZZ_{|V|-1}^{(m-1)|V|}.
$$
\end{proposition}

\begin{proof}
Again check Corollary~\ref{stacked-bipartite-corollary} applies.
The vertex degrees $(d_1,\ldots,d_{|V|})$ are all $|V|-1$, 
relatively prime to $m$.
The adjacency matrix for $G$ is 
$J - I$ of the $|V| \times |V|$ all ones matrix $J$ and the 
identity matrix $I$, with 
eigenvalues $(|V|,0,0,\ldots,0)-(1,1,1,\ldots,1)=
(|V|-1,-1,-1,\ldots,-1)$, and hence determinant 
$\pm (|V|-1)$, relatively prime to $m$.
Having checked the corollary applies, one needs
the well-known fact that
$K(G)=\ZZ_{|V|}^{|V|-2}$, with invariant factors
$(k_1,\ldots,k_{|V|-1})=(1,|V|,|V|,\ldots,|V|)$.
\end{proof}

\noindent
We remark that that in this example, the graph $\widetilde{mG}$ is the {\it complete $|V|$-partite graph}
$K_{m,m,\ldots,m}$, whose critical group was computed for all $|V|$ and $m$ in \cite[Cor. 5]{Jacobson}.
One can check that the answer given in Proposition~\ref{stacked-bipartite-completes}
agrees with this computation when $|V|-1$ and $m$ are relatively prime.

\section{Signed graphs in general:  allowing half-loops}
\label{general-signed-graph-section}

This section reviews the more general notion of signed graphs $G_\pm$, as in Zaslavsky \cite{Zaslavsky}, 
in which one allows positive and negative {\it half-loops}, with the goal of
generalizing our definition of the critical group $K(G_\pm)$ to this case.
This gives us the flexibility to consider in the
next section a more general notion of double covering, both for unsigned and signed graphs, 
in which a half-loop can be doubly covered by a single edge.  

For example, in 
Section~\ref{Bai-section} we will use this to re-interpret a calculation of H. Bai on the 
critical group of the $n$-dimensional cube graph $Q_n$:  the obvious projection $Q_n \rightarrow Q_{n-1}$ can be regarded as such a double cover, in which each edge of $Q_n$ parallel to the 
direction of projection doubly covers a half-loop added to its image vertex in $Q_{n-1}$; see Figure~\ref{cube-squashing-double-cover-figure}.

\subsection{Definition of a general signed graph critical group}

\begin{definition}
An unsigned {\it multigraph with half-loops} $G=(V,E)$ is a multigraph in which some of the self-loops have been designated as {\it half-loops}.  A {\it signed graph} $G_{\pm}$ consists of an underlying multigraph with half-loops
$G=(V,E)$ together with an 
assignment $\beta: E \rightarrow \{+1,-1\}(=:\{+,-\})$, 
designating edges positive or negative.
\end{definition}

For these more general signed graphs, we will need two closely related versions of an node-edge-incidence matrix, 
$\del=\del_{G_\pm}$ and $\delT=\delT_{G_\pm}$, both lying in $\ZZ^{V \times E}$, that is,
both regarded as $\ZZ$-linear maps $\ZZ^E \rightarrow \ZZ^V$.  As before, one first chooses
an arbitrary orientation of the edges $E$ to write them down.

\begin{definition}
The map $\del$ treats loops and half-loops the same, sending 
an edge $e$ directed from $u$ to $v$ to $+u -\beta(e)v$, even if $u=v$.
This means that $\del$ sends positive loops and positive half-loops to $0$,
and sends both a negative loop and negative half-loop on vertex $v$ to $+2v$.

The map $\delT$ is almost the same, except that it treats negative loops and negative half-loops 
unequally.  Just as with $\del$, the map $\delT$ sends an edge $e$ 
directed from $u$ to $v$ to $+u -\beta(e)v$ when $u \neq v$.
Also just as with $\del$, the map $\delT$ send both positive loops and positive half-loops to $0$,
and $\delT$ sends a negative loop on vertex $v$ to $+2v$.  
However, $\delT$ sends a negative {\it half-loop} on vertex $v$ to $+v$.
\end{definition}

\begin{remark}
The map $\delT$ is the signed graph incidence matrix used by Zaslavsky in \cite[\S 8A]{Zaslavsky}.
Note that $\delT=\del$ if and only if $G_\pm$ contains no negative half-loops.
\end{remark}

\begin{definition}
For a signed graph $G_\pm$, define its {\it critical group}
\begin{align}
K(G_\pm) & := \im \del / \im \del \delT^t 
\label{vertex-presentation-for-general-signed-Laplacian}\\
& \cong \ZZ^E / \left( \im \delT^t + \ker \del \right) \quad \text{ via Proposition~\ref{concordance-prop}}. 
\label{edge-presentation-for-general-signed-Laplacian}
\end{align}
where we will call the matrix 
$L(G_{\pm}) := \del\delT^t$ appearing above a {\it signed graph Laplacian}.
\end{definition}

\subsection{Issues of well-definition}

Note that this definition of $K(G_\pm)$ generalizes our 
earlier definition for the more restrictive signed graphs 
in the Introduction, where half-loops were disallowed. 
The next proposition answers some other obvious questions 
which are not as familiar or transparent as for unsigned graphs.

\begin{proposition}
The signed graph Laplacian matrix $L(G_{\pm}) = \del\delT^t$ has entries
$$
L(G_{\pm})_{u,v}=
\begin{cases}  
\# \{ \text{negative edges with endpoints }u,v \} & \\
 \quad -\# \{ \text{positive edges with endpoints }u,v \} &\text{ if }u \neq v,\\
 & \\
\#\{\text{non-loop edges (positive or negative) incident to }v\} & \\
 \quad  +4\#\{\text{negative (full) loops at }v\} + 2\#\{\text{negative half-loops at }v\} & \text{ if }u=v.
\end{cases}
$$
In particular, 
\begin{itemize}
\item $L(G_{\pm})$ is symmetric, and
\item both the matrix $L(G_\pm)$ and the isomorphism type of the
abelian group $K(G_\pm)$ do not depend upon the 
choice of orientation of the edges $E$ 
used to write down $\del$ and $\delT$.
\end{itemize}

\end{proposition}
\begin{proof}
The matrix entry calculation for $L(G_\pm)$ is straightforward, 
and does not depend on the orientations.  
Note also that the sublattice $\im \del$ inside $\ZZ^V$ 
does not depend upon the orientations, as
a typical column of $\im \del$ for an oriented edge 
$e=(u,v)$ is $\del(e)=+u-\beta(e)v=\pm (+v - \beta(e) u)$.  
Thus $K(G_\pm)=\im \del / \im \del \delT^t$ does not change when one reorients edges.
\end{proof}

It is fairly obvious for unsigned graphs that the critical group $K(G)$ is an isomorphism invariant of the graph $G=(V,E)$,
since permuting or relabelling vertices corresponds to permuting
the coordinates of the ambient space 
$\RR^V \supset \ZZ^V \supset \im \del \supset \im L(G)$,
without altering $K(G)$ up to isomorphism. 
Of course, the same holds for permutation of the vertices in
signed graphs $G_\pm$.  However, there is a stronger
notion of {\it signed graph isomorphism} that allows not only permuting or
relabelling vertices, but in addition,
at any vertex $v$ in $V$ one can perform the {\it switch} at $v$ 
(see \cite[\S 3]{Zaslavsky}) on $G_\pm$, 
which has the effect of exchanging $\beta(e)$ via $+ \leftrightarrow -$ 
for every non-loop, non-half-loop edge $e$ incident to $v$.
Algebraically, this corresponds to a sign change in the $v$-coordinate
of the ambient space $\RR^V$, and again does not alter $K(G)$ up to isomorphism.
We will take advantage of such signed graph isomorphisms in the next section.

\subsection{Balanced cycles and the image of $\del$}

The following simple notion is an important signed graph isomorphism invariant
(see \cite[\S 2]{Zaslavsky}), dictating the nature of $\im \del$ inside $\ZZ^V$.

\begin{definition}
For a signed graph $G_\pm$ with underlying multigraph $G=(V,E)$,
consider a subset $C \subset E$ forming a cycle in $G$, with $C$ possibly 
a singleton (full) loop or half-loop.  Call $C$ a {\it balanced } (resp. {\it unbalanced}) {\it cycle} of $G_\pm$ 
if the number of negative edges $e$ in $C$ (that is,
those with $\beta(e)=-$) is even (resp. odd).
\end{definition}

For unsigned graphs, the description of the 
sublattice $\im \del$ inside $\ZZ^V$ is fairly straightforward:
when $G=(V,E)$ has connected components with vertex sets 
$V_1,V_2,\ldots,V_t$, one has compatible direct sum decompositions
$
\ZZ^V = \bigoplus Z^{V_i}
$
and 
$
\im \del = \bigoplus_{i=1}^t \ZZ^{V_i}_{=0}
$
where $\ZZ^V_{=0}:=\{x \in \ZZ^V: \sum_{v \in V} x_v =0\}$.

For a signed graph $G_\pm$ one again has the same reduction to each connected component of its underlying multigraph, which one can therefore assume is connected.  

\begin{proposition}
\label{connected-signed-graph-im-of-del}
A signed graph $G_\pm$ with connected underlying multigraph has
two cases for $\im \del$:
\begin{enumerate}
\item[(i)] $\im \del = \ZZ^V_{\equiv 0 \bmod 2} :=\{x \in \ZZ^V: \sum_{v \in V} x_v \equiv 0 \bmod 2\}$
if $G_\pm$ contains at least one unbalanced cycle;  call $G_\pm$
{\bf unbalanced} in this case.
\item[(ii)] $\im \del = \ZZ^V_{=0}$ if $G_\pm$ has no unbalanced cycles;
call $G_\pm$ {\bf balanced} in this case.
\end{enumerate}

Moreover, $G_\pm$ being balanced 
is equivalent to it being signed-graph isomorphic
to an unsigned multigraph with half-loops, that is, 
a signed graph with no negative edges.
\end{proposition}
\begin{proof}
Straightforward, or see Zaslavsky \cite[Prop. 2.1, Thm. 5.1]{Zaslavsky}.
\end{proof}

\subsection{The cardinality of the critical group}
\label{signed-matrix-tree-subsection}

The definition of the signed graph critical group $K(G_\pm)$ 
does not make it clear that
it is a {\it finite} group.  We pause here
to show this, and to give a signed
generalization of the formula for the
cardinality of unsigned graph critical groups in terms of
maximal forests.  
The methodology is straightforward, proven analogously
to Zaslavsky's Matrix-Tree Theorem for signed graphs \cite[Thm. 8A.4]{Zaslavsky}, via the Binet-Cauchy Theorem\footnote{Our answer differs somewhat from
Zaslavsky's because he dealt with a Laplacian
matrix of the form $\delT \delT^t$ and computed the cardinality of
$\ZZ^V / \im \delT \delT^T$, 
whereas we deal with our Laplacian $L(G)= \del \delT^t$
and compute the cardinality of $K(G) = \im \del / \del \delT^t$.}.

To start, one needs to know the analogue of maximal 
forests in unsigned graphs.

\begin{proposition} 
\label{signed-base-description}
Given a signed graph $G_\pm$ having underlying multigraph $G=(V,E)$, consider
a subset $B \subset E$.  Then $B$ indexes a subset of columns of $\del$ or $\delT$ forming a basis for the full column space if and only
\begin{enumerate}
\item[$\bullet$] its intersection with
each balanced connected component of $G$ forms a spanning tree, and 
\item[$\bullet$]  its intersection with
each unbalanced connected component of $G$ 
is a collection of unicyclic connected components,
(that is, each connected component contains a unique cycle) and the
unique cycle of each component is unbalanced.
\end{enumerate}
\end{proposition}
\begin{proof}
Straightforward, or see Zaslavsky \cite[Thm. 5.1(g)]{Zaslavsky}.
\end{proof}

\noindent
With this in hand, we will be able to
describe the cardinality of the critical group $K(G_\pm)$
as a sum over such bases $B \subset E$.  Given such a base $B$ as described in
Proposition~\ref{signed-base-description}, define the quantity
$d(B)$ to be a product over the connected components $B_1,\ldots,B_t$
induced by the edges of $B$, where
\begin{enumerate}
\item[$\bullet$]
a component $B_i$ forming a spanning tree for a balanced component of $G$ 
contributes a factor of $1$,
\item[$\bullet$]
a component $B_i$ which is unicyclic will either contribute a factor of $2$ (resp. $4$)
if its unique cycle is a singleton negative half-loop (resp. 
is an unbalanced cycle that contains no negative half-loop).
\end{enumerate}

\begin{proposition}
\label{matrix-tree-like-prop}
A signed graph $G_\pm$ having $c$ unbalanced connected components will have 
$$
|K(G_\pm)| = 2^{-c} \sum_{B} d(B),
$$
where $B$ runs over the bases in Proposition~\ref{signed-base-description}.
In particular, $K(G)$ is finite.

\end{proposition}
\begin{proof}
Note that by the definition of $d(B)$, the quantity $F(G_\pm)$ 
on the right side of the proposition which we wish to show equals $|K(G_\pm)|$ 
has the multiplicative property that
$F(G_\pm) = \prod_{i=1}^t F(G_\pm^i)$
if the underlying multigraph $G_{\pm}=(V,E)$ has connected components 
$G^1_{\pm},\ldots,G^t_{\pm}$.  The compatible direct sum decompositions
$
\ZZ^V = \bigoplus_i \ZZ^{V^i}
$
and
$
K(G_\pm) \cong \bigoplus_i K(G^i_{\pm})
$
imply that $|K(G_\pm)|$ has this same multiplicative
property, so it suffices to prove the proposition when $G_\pm$ is connected.

When $G_\pm$ is connected and balanced, one can assume after applying a signed graph isomorphism, that $G_\pm$ is an unsigned graph.  Then
$K(G_\pm)$ is the usual critical group, which is finite, and has  
cardinality equal to 
the number of spanning trees, which agrees with $F(G_\pm)$.

When $G_\pm$ is connected and unbalanced, we calculate
$\det L(G_\pm)$ explicitly and show that it is positive: this will 
in particular show that $K(G_\pm):=\im \del/\im L(G_\pm)$ is finite, 
since it implies $\im L(G_\pm)$ has full rank inside
$\ZZ^V$ and hence also inside $\im \del$.
One starts by using the Binet-Cauchy Theorem 
to express $\det L(G_\pm)$ as a sum over bases $B$:
$$
\begin{aligned}
\det L(G_\pm)=\det \del \delT^t 
&=\sum_B \det \del|_{\text{cols }B} \cdot \det \delT^t|_{\text{rows }B}  
=\sum_B \det \del|_{\text{cols }B} \cdot \det \delT|_{\text{cols }B} \\
&=\sum_B \prod_{\substack{\text{connected}\\\text{components }\\B_i\text{ of }B}} 
 \det \del|_{\text{cols }B_i} \cdot \det \delT|_{\text{cols }B_i} 
=\sum_B d(B)
\end{aligned}
$$
The last equality used the following calculation,
which one can reduce to the case where $B_i$ is an unbalanced cycle,
via an induction that plucks off leaf vertices
(vertices with only one incident edge):

\begin{align}
\det \del|_{\text{cols B}_i} & = \pm 2,\label{del-cases-of-plus-minus-2}\\
\det \delT|_{\text{cols }B_i} &= 
\begin{cases}
\pm 1 & \text{ if its unbalanced cycle is a negative half-loop,}\\
\pm 2 & \text{ otherwise. } 
\end{cases},\label{delT-cases-of-plus-minus-2}
\end{align}
so that one has
\begin{equation}
\label{products-are-plus-twos-and-fours}
\del|_{\text{cols B}_i} \cdot \det \delT|_{\text{cols }B_i}
=
\begin{cases}
+2 & \text{ if its unbalanced cycle is a negative half-loop,}\\
+4 & \text{ otherwise. } 
\end{cases}
\end{equation}
A crucial point to be emphasized here is 
that signs on the $+2$ and $+4$ are always positive
in \eqref{products-are-plus-twos-and-fours}
because the plus/minus signs on the $\pm 1, \pm 2$
{\it always agree in } 
\eqref{del-cases-of-plus-minus-2}
and \eqref{delT-cases-of-plus-minus-2}:
these signs will be 
determined by the choices of orientations of the edges in $B_i$.
This shows 
\begin{equation}
\label{signed-Laplacian-nonsingular}
|\ZZ^V / \im L(G_\pm)|=\det L(G_\pm)=\sum_B d(B) > 0
\end{equation}
for connected unbalanced signed graphs $G_\pm$.  
But $F(G_\pm)=\frac{1}{2}\sum_B d(B)$ in this case,
in agreement with 
$$
|K(G)|=|\im \del/\im L(G_\pm)|= \frac{1}{2}|\ZZ^V / \im L(G_\pm)|
$$
where the $\frac{1}{2}$ arises here since 
$\im \del$ is the index two sublattice 
$\ZZ^V_{\equiv 0 \bmod 2}$ inside $\ZZ^V$.
\end{proof}

\section{Doubly covering a signed graph}
\label{signed-graph-double-cover-section}

Our goal in this section is to define a
notion of a {\it signed graph double coverings}, 
leading to a more flexible 
generalization of 
Theorem~\ref{unsigned-double-cover-theorem}.

\subsection{The double cover construction for signed graphs}

\begin{definition}
\label{signed-graph-double-cover-definition}
Given two signed graphs $G^{(i)}_\pm$ for $i=1,2$ with same underlying multigraph
$G=(V,E)$, and edge orientation on $E$ chosen arbitrarily,
define a signed graph 
$$
\tilde{G}_\pm :=\Double( G^{(1)}_\pm, G^{(2)}_\pm ),
$$
which we will think of as a {\it double cover of the base $G^{(1)}_\pm$ 
parametrized by the voltage-assignment signed graph $G^{(2)}_\pm$}.
It has vertex set 
$
\tilde{V}:=\{v_+,v_-\}_{v \in V}
$
and edge set $\tilde{E}$ defined and oriented as follows.
For each edge $e=(u,v)$ in $G$ which is not a half-loop,
(so possibly $u=v$ if $e$ is a full loop), 
create two full (directed) edges $e_+,e_-$ of $\tilde{G}_\pm$  
having the same sign as $e$ {\bf in }$G^{(1)}_\pm$, with these endpoints:
$$
\begin{cases}
e_+=(u_+,v_+), e_-=(u_-,v_-) & 
 \text{ if }G^{(1)}_\pm, G^{(2)}_\pm\text{ agree on the sign of }e,\\
e_+=(u_+,v_-), e_-=(u_-,v_+) & 
 \text{ if }G^{(1)}_\pm, G^{(2)}_\pm\text{ disagree on the sign of }e.
\end{cases}
$$
For each half-loop edge $e$ at vertex $v$ in $G$ 
create either one or two edges of $\tilde{G}_\pm$ 
having the same sign as $e$ {\bf in }$G^{(1)}_\pm$, with these endpoints:
$$
\begin{cases}
\text{ half-loops }e_+=(v_+,v_+), e_-=(v_-,v_-) & 
 \text{ if }G^{(1)}_\pm, G^{(2)}_\pm\text{ agree on the sign of }e,\\
\text{ edge }\tilde{e}=(v_+,v_-)& 
 \text{ if }G^{(1)}_\pm, G^{(2)}_\pm\text{ disagree on the sign of }e.
\end{cases}
$$
\end{definition}

\noindent
Figure~\ref{cube-squashing-double-cover-figure} depicts an example,
first showing 
$\tilde{G} =\Double( G^{(1)}_\pm, G^{(2)}_\pm ) \rightarrow G^{(1)}_\pm$,
and then below it $G^{(2)}_\pm$.

\begin{figure}
\label{cube-squashing-double-cover-figure}
\begin{center}
\begin{tikzpicture}[->,scale=0.8, auto,swap]
    \foreach \pos/\name in {{(-5,-9)/a},{(-2,-9)/b},{(-2,-6)/c},{(-5,-6)/d}}
        \node[vertex] (\name) at \pos {$\name$};
    \foreach \source/\dest/\sign in {a/b/+,b/c/+,c/d/+,d/a/+}
        \path (\source) edge node[auto] {$\sign$} (\dest);
\foreach \source in {a,b,c,d}    
	\path[every node/.style={font=\sffamily\small}]
        (\source)   edge[in=10, out=70, loop] node[auto] {$-$}  (\source);
\foreach \source in {a,b,c,d}    
	\path[every node/.style={font=\sffamily\small}]
        (\source)   edge[in=110, out=170, loop] node[auto] {$-$}  (\source);
\foreach \source in {a,b,c,d}    
	\path[every node/.style={font=\sffamily\small}]
        (\source)   edge[in=190, out=260, loop] node[auto] {$+$}  (\source);

\draw[arrows=->,line width=1pt](-3.5,-3)--(-3.5,-4.5);

    \foreach \pos/\name in {{(-5,0)/a_+},{(-5,3)/b_+},{(-2,3)/c_+},{(-2,0)/d_+},
                            {(-6.5,-1.5)/a_-},{(-6.5,1.5)/b_-},{(-3.5,1.5)/c_-},{(-3.5,-1.5)/d_-}}
        \node[vertex] (\name) at \pos {$\name$};
    \foreach \source/ \dest/ \sign in {b_+/c_+/+,d_+/a_+/+,b_-/c_-/+,d_-/a_-/+}
        \path (\source) edge node[above, near start] {$\sign$} (\dest);
     \foreach \source/ \dest/ \sign in {a_+/b_+/+,c_+/d_+/+,a_-/b_-/+,c_-/d_-/+,a_+/a_-/+,b_+/b_-/+,c_+/c_-/+,d_+/d_-/+}
        \path (\source) edge node[near start, left] {$\sign$} (\dest);

\foreach \source in {a_+,a_-,b_+,b_-,c_+,c_-,d_+,d_-}    
	\path[every node/.style={font=\sffamily\small}]
        (\source)   edge[in=110, out=170, loop] node[auto] {$-$}  (\source);
\foreach \source in {a_+,a_-,b_+,b_-,c_+,c_-,d_+,d_-}    
	\path[every node/.style={font=\sffamily\small}]
        (\source)   edge[in=10, out=70, loop] node[auto] {$-$}  (\source);
\end{tikzpicture}

\begin{tikzpicture}[->,scale=0.8, auto,swap]
    \foreach \pos/\name in {{(0,0)/a},{(0,3)/b},{(3,3)/c},{(3,0)/d}}
        \node[vertex] (\name) at \pos {$\name$};
    \foreach \source/\dest/\sign in {a/b/+,b/c/+,c/d/+,d/a/+}
        \path (\source) edge node[auto] {$\sign$} (\dest);
\foreach \source in {a,b,c,d}    
	\path[every node/.style={font=\sffamily\small}]
        (\source)   edge[in=10, out=70, loop] node[auto] {$-$}  (\source);
\foreach \source in {a,b,c,d}    
	\path[every node/.style={font=\sffamily\small}]
        (\source)   edge[in=110, out=170, loop] node[auto] {$-$}  (\source);
\foreach \source in {a,b,c,d}    
	\path[every node/.style={font=\sffamily\small}]
        (\source)   edge[in=190, out=260, loop] node[auto] {$-$}  (\source);
\end{tikzpicture}

\caption{
An example of a signed graph double covering 
$$
\tilde{G}_\pm := 
\Double(G^{(1)}_\pm, G^{(2)}_\pm) \overset{\pi^{(1)}}{\longrightarrow} G^{(1)}
$$
and the signed graph $G^{(2)}$,
in which all of the loop edges shown are intended to be half-loops.
Note that the half-loop edges $e$ in the underlying graph $G$ where $G^{(1)}, G^{(2)}$
disagree on their $+/-$ voltage assignment are ``doubly covered'' under the projection by the edges of $\tilde{G}$ that point in the third coordinate
direction.}

\end{center}
\end{figure}

\subsection{Properties of signed graph double coverings}

As anticipated in the phrasing of 
Definition~\ref{signed-graph-double-cover-definition}, 
we will speak of a double-covering map 
$$
\begin{array}{rcl}
\tilde{G}_\pm & \overset{\pi_{(1)}}{\longrightarrow} &G^{(1)}_\pm\\
v_+, v_-  &\longmapsto &v\\
e_+, e_- &\longmapsto &e\\
\end{array}
$$
as also ``doubly covering'' each half-loop $e=(v,v)$
with opposite signs in $G^{(1)}_\pm, G^{(2)}_\pm$ via
$
\tilde{e}=(v_+,v_-) \longmapsto e.
$

We first note that it generalizes the unsigned graph double coverings defined earlier.  The proof of the following
proposition is a straightforward exercise in the definitions.

\begin{proposition}
Let 
$
\tilde{G}_\pm =\Double( G^{(1)}_\pm, G^{(2)}_\pm ),
$
be a signed graph double covering, in which 
$G^{(1)}_\pm=G$ is an unsigned multigraph with no half-loops, meaning
that $G^{(1)}$ has $\beta(e)=+$ and has no half-loops.

Then $\tilde{G}_\pm=\tilde{G}$ 
is also an unsigned multigraph with no half-loops,
and $\tilde{G}_\pm \overset{\pi_{(1)}}{\longrightarrow} G^{(1)}_\pm$
is the same as the (regular) double covering $\tilde{G} \overset{\pi}{\rightarrow} G$
corresponding to the voltage graph $G^{(2)}_\pm$. \qedhere
\end{proposition}

\noindent
The asymmetry of the roles of 
$G^{(1)}_\pm$ and $G^{(2)}_\pm$ in constructing 
$\Double(G^{(1)}_\pm, G^{(2)}_\pm)$ turns out to
be illusory.

\begin{proposition}
\label{asymmetry-illusory-prop}
The two signed graphs 
$\Double(G^{(2)}_\pm, G^{(1)}_\pm)$
and 
$\Double(G^{(1)}_\pm, G^{(2)}_\pm)$
are actually signed graph isomorphic:  one is obtained from the
other by switching at every vertex in the subset
$\{v_-\}_{v \in V}$ of their common vertex set $\tilde{V}=\{v_+,v_-\}_{v \in V}$.
\end{proposition}

\begin{proof}
Let $\tilde{G}_\pm:=\Double(G^{(1)}_\pm, G^{(2)}_\pm)$,
and let $\tilde{G}'_\pm$ be the result of performing
the signed isomorphisms described in the proposition.
Then the edges whose voltage $+/-$ signs will have changed from
$\tilde{G}_\pm$
to $\tilde{G}'_\pm$
are the edges that cross the
vertex cut from $\{v_-\}_{v \in V}$ to $\{v_+\}_{v \in V}$.  
These are exactly the edges
of $G$ whose voltage signs in $G^{(1)}_\pm,G^{(2)}_\pm$ disagreed, 
so that in $\tilde{G}'_\pm$ they carry voltages that
agree with $G^{(2)}_\pm$.  
The remaining edges in
$\tilde{G}_\pm$ already agreed in voltage with $G^{(2)}_\pm$, so 
{\it all} edges of $\tilde{G}'_\pm$ agree with
$G^{(2)}_\pm$.  In addition, those edges of 
$\tilde{G}'_\pm$ which cross the vertex cut 
will still be the ones where the voltages on $G^{(1)}_\pm, G^{(2)}_\pm$ 
disagree.  Thus $\tilde{G}'_\pm$ matches the description of $\Double(G^{(2)}_\pm, G^{(1)}_\pm)$.
\end{proof}

\noindent
This hidden symmetry between 
 the ``base graph'' $G^{(1)}_\pm$ and
``voltage assignment'' $G^{(2)}_\pm$
becomes apparent only after generalizing graph double covers to signed
graphs, and was one motivation for introducing such covers.

\section{The short complex for a double covering of signed graphs}
\label{short-complex-section}

Our goal here is a second generalization of 
Theorem~\ref{unsigned-double-cover-theorem}
that applies to signed graph double covers $\Double(G^{(1)}_\pm, G^{(2)}_\pm)$.
When working with these signed graph 
critical groups $K(G_\pm)$, we 
could in principle use the edge-presentation
\eqref{edge-presentation-for-general-signed-Laplacian}
as $K(G_\pm)=\ZZ^E / (\im \delT^t + \ker \del)$.
However, we have found it more convenient in the proofs
of this section to work with the vertex-presentation
\eqref{vertex-presentation-for-general-signed-Laplacian}:
$$
K(G_\pm) = \im \del / \im \del \delT^t
=\im \del / L(G_\pm).
$$
Thus we define various maps on the level of 
the vertex groups $\ZZ^V, \ZZ^{\tilde{V}}$,
inducing morphisms of critical groups.

\begin{definition}
Given $\tilde{G}_\pm =\Double(G^{(1)}_\pm, G^{(2)}_\pm)$, as before, 
consider free $\ZZ$-modules $\ZZ^V, \ZZ^{\tilde{V}}$ and
$\ZZ^E,\ZZ^{\tilde{E}}$,
having $\ZZ$-basis elements indexed by vertices or edges 
in sets $V,\tilde{V}$ and $E, \tilde{E}$. 

On the level of vertices, define $\ZZ$-linear maps
$$
\begin{array}{rcccccl}
\ZZ^{\tilde{V}} & \overset{\pi_{(1)}}{\longrightarrow} & \ZZ^V & \qquad & 
\ZZ^{V} & \overset{\pi^t_{(1)}}{\longrightarrow} & \ZZ^{\tilde{V}}\\
v_+ & \longmapsto  & +v &    & v & \longmapsto  & v_++v_-\\
v_- & \longmapsto  & +v &    &   &              & \\
 & & & & & & \\
\ZZ^{\tilde{V}} & \overset{\pi_{(2)}}{\longrightarrow} & \ZZ^V & \qquad & 
\ZZ^{V} & \overset{\pi^t_{(2)}}{\longrightarrow} & \ZZ^{\tilde{V}}\\
v_+& \longmapsto  & +v&      & v & \longmapsto  & v_+-v_-\\
v_-& \longmapsto  & -v&      &   &              & 
\end{array}
$$
\noindent
Also define an involution 
$$
\begin{array}{rcl}
\ZZ^{\tilde{V}} & \overset{\iota}{\longrightarrow} & \ZZ^{\tilde{V}} \\
v_+ & \longmapsto & v_- \\
v_- & \longmapsto & v_+ 
\end{array}
$$
and these two sublattices of $\ZZ^{\tilde{V}}$ 
$$
\begin{aligned}
\ZZ^{\tilde{V}}_\sym &:=\{x \in \ZZ^{\tilde{V}}: \iota(x) = x\},\\
\ZZ^{\tilde{V}}_\antisym &:=\{x \in \ZZ^{\tilde{V}}: \iota(x) = -x\}.
\end{aligned}
$$
\end{definition}

We collect in the next proposition the various necessary
technical properties of these maps $\pi_{(i)}$ and $\iota$.
\begin{proposition}
\label{technical-proposition}
Given $\tilde{G}_\pm =\Double(G^{(1)}_\pm, G^{(2)}_\pm)$, one has the
following properties of $\pi_{(i)}$ for $i=1,2$.
\begin{enumerate}
\item[(i)] $\pi_{(i)}( \im \del_{\tilde{G}_\pm} ) = \im \del_{G^{(i)}_\pm}$.
\item[(ii)] $\pi^t_{(i)}( \im \del_{G^{(i)}_\pm} ) \subset \im \del_{\tilde{G}_\pm}$.
\item[(iii)] $\pi_{(i)} ( \im L(\tilde{G}_\pm)  )
                =\im L(G^{(i)}_\pm ).$
\item[(iv)] 
This diagram commutes
$$
\begin{CD}
\ZZ^{\tilde{V}} @>{ L(\tilde{G}_\pm)}>>  \ZZ^{\tilde{V}} \\
@A{  \pi^t_{(i)} }AA  
    @AA{ \pi^t_{(i)} }A \\
\ZZ^V @>{  L(G^{(i)}_\pm) }>> \ZZ^V
\end{CD}
$$
\item[(v)] These two sequences are short exact:
$$
\begin{array}{rccl}
0 \longrightarrow &\ZZ^V 
   \overset{\pi_{(1)}^t}{\longrightarrow} &\ZZ^{\tilde{V}}
    \overset{\pi_{(2)}}{\longrightarrow} &\ZZ^V 
      \longrightarrow 0 \\
0 \longrightarrow &\ZZ^V 
   \overset{\pi_{(2)}^t}{\longrightarrow} &\ZZ^{\tilde{V}}
    \overset{\pi_{(1)}}{\longrightarrow} &\ZZ^V 
      \longrightarrow 0 
\end{array}
$$
with 
$$
\begin{array}{rcl}
\im \pi_{(1)}^t &= \ker \pi_{(2)} &= \ZZ^{\tilde{V}}_\sym \\
\im \pi_{(2)}^t &= \ker \pi_{(1)} &= \ZZ^{\tilde{V}}_\antisym \\
\end{array}
$$
\item[(vi)] As operators on $\ZZ^{\tilde{V}}$, the
map $\iota$ commutes with $L(\tilde{G}_\pm)$.
\end{enumerate}
\end{proposition}

\begin{proof}
By Proposition~\ref{asymmetry-illusory-prop}, it 
suffices to check the assertions for $\pi_{(1)}$; the assertions for 
$\pi_{(2)}$ will then follow by applying sign switches at all
vertices $\{v_-\}_{v \in V}$ of $\tilde{G}_\pm$.

In proving assertions (i),(ii),(iii),(iv),
it is convenient to introduce two maps
$\ZZ^{\tilde{E}} \overset{\pi_{(1)}}{\underset{\rho_{(1)}}{\rightleftarrows}} \ZZ^E$
defined as follows:
$$
\begin{array}{rccl}
\ZZ^{\tilde{E}} & \overset{\pi_{(1)}}{\longrightarrow} & \ZZ^E& \\
e_+,e_- & \longmapsto  & e &\text{ if }e\text{ has two preimages }e_+,e_i\text{ in }\tilde{E} \\
\tilde{e}
  & \longmapsto  & e &\text{ if }e\text{ is a half-loop with preimage }\tilde{e},\text{ and voltages }\beta(e)=-1,+1\text{ in }G^{(1)}_\pm,G^{(2)}_\pm,\text{ resp.}\\
\tilde{e}  & \longmapsto  & 0 &\text{ if }e\text{ is a half-loop with preimage }\tilde{e},\text{ and voltages }\beta(e)=+1,-1\text{ in }G^{(1)}_\pm,G^{(2)}_\pm,\text{ resp.}\\
 & & & \\
\ZZ^{E} & \overset{\rho_{(1)}}{\longrightarrow} & \ZZ^{\tilde{E}}& \\
e & \longmapsto  & e_++e_- &\text{ if }e\text{ has two preimages }e_+,e_i\text{ in }\tilde{E} \\
e
  & \longmapsto  & 2\tilde{e} &\text{ if }e\text{ is a half-loop with preimage }\tilde{e},\text{ with voltages }\beta(e)=-1,+1\text{ in }G^{(1)}_\pm,G^{(2)}_\pm,\text{ resp.}\\
e  & \longmapsto  & 0 &\text{ if }e\text{ is a half-loop with preimage }\tilde{e},\text{ with voltages }\beta(e)=+1,-1\text{ in }G^{(1)}_\pm,G^{(2)}_\pm,\text{ resp.}
\end{array}
$$
Note that $\ZZ^E \overset{\rho_{(1)}}{\rightarrow} \ZZ^{\tilde{E}}$ is close, 
but not quite equal, to the transpose $\pi_{(1)}^t$ of the
map $\ZZ^{\tilde{E}} \overset{\pi_{(1)}}{\rightarrow} \ZZ^E$.
These maps $\pi_{(1)}, \rho_{(1)}$ 
between edge lattices correspond to the maps
$\pi_{(1)}, \pi^t_{(1)}$ already defined between vertex lattices,
in the sense that one has these easily-checked commutative diagrams:
\begin{equation}
\label{edge-lift-of-pi-commutative-diagrams}
\begin{CD}
\ZZ^{\tilde{V}} @>{ \delta^t_{\tilde{G}_\pm} }>>  \ZZ^{\tilde{E}} \\
@V{  \pi_{(1)} }VV  
    @VV{ \pi_{(1)} }V \\
\ZZ^V @>{  \delta^t_{G^{(1)}_\pm } }>> \ZZ^E
\end{CD}
\qquad \qquad \qquad
\begin{CD}
\ZZ^{\tilde{E}} @>{ \partial_{\tilde{G}_\pm} }>>  \ZZ^{\tilde{V}} \\
@V{  \pi_{(1)} }VV  
    @VV{ \pi_{(1)} }V \\
\ZZ^E @>{  \partial_{G^{(1)}_\pm } }>> \ZZ^V
\end{CD}
\end{equation}
and
\begin{equation}
\label{edge-lift-of-pi-transpose-commutative-diagrams}
\begin{CD}
\ZZ^{\tilde{V}} @>{ \delta^t_{\tilde{G}_\pm} }>>  \ZZ^{\tilde{E}} \\
@A{  \pi^t_{(1)} }AA  
    @AA{ \rho_{(1)} }A \\
\ZZ^V @>{  \delta^t_{G^{(1)}_\pm } }>> \ZZ^E
\end{CD}
\qquad \qquad \qquad
\begin{CD}
\ZZ^{\tilde{E}} @>{ \partial_{\tilde{G}_\pm} }>>  \ZZ^{\tilde{V}} \\
@A{  \rho_{(1)} }AA  
    @AA{ \pi^t_{(1)} }A \\
\ZZ^E @>{  \partial_{G^{(1)}_\pm } }>> \ZZ^V.
\end{CD}
\end{equation}

\vskip.1in
\noindent
{\sf Assertion (ii).}
This follows immediately from the right commutative square
in \eqref{edge-lift-of-pi-transpose-commutative-diagrams}.

\vskip.1in
\noindent
{\sf Assertion (i).}
The weaker inclusion 
$\pi_{(1)}( \im \del_{\tilde{G}_\pm} ) \subseteq \im \del_{G^{(1)}_\pm}$
similarly follows immediately from the right
commutative square in \eqref{edge-lift-of-pi-commutative-diagrams}.
One wants to show that this inclusion is an equality, which would
follow if $\ZZ^{\tilde{E}} \overset{\pi_{(1)}}{\longmapsto} \ZZ^{E}$ 
were surjective.
Although it is {\it not} necessarily surjective,  
the only basis elements in $\mb{Z}^{E}$ not in the
image of $\pi_{(1)}$ correspond to half loops $e$ 
which are positive in $G^{(1)}_\pm$ and are covered by a single edge in $\tilde{G}_\pm$,
and these elements lie in the kernel of $\partial_{G^{(1)}_\pm }$.
Therefore $\partial_{G^{(1)}_\pm }(\im \pi_{(1)})=\partial_{G^{(1)}_\pm }(\ZZ^E)$, 
and the equality follows.

\vskip.1in
\noindent
{\sf Assertion (iv).}
This follows immediately from the commutative square 
$$
\begin{CD}
\ZZ^{\tilde{V}} @>{ L(\tilde{G}_\pm) }>>  \ZZ^{\tilde{V}} \\
@A{  \pi^t_{(1)} }AA  
    @AA{ \pi^t_{(1)} }A \\
\ZZ^V @>{  L(G^{(1)}_\pm) }>> \ZZ^V
\end{CD}
$$
obtained by gluing the two commutative squares
in \eqref{edge-lift-of-pi-transpose-commutative-diagrams} along their
common vertical edge $\ZZ^{E} \overset{\rho_{(1)}}{\rightarrow} \ZZ^{\tilde{E}}$, 
and then composing the horizontal maps in the top
and bottom rows.

\vskip.1in
\noindent
{\sf Assertion (iii).}
The weaker inclusion 
$\pi_{(1)} ( \im L(\tilde{G}_\pm)  )
\subseteq \im L(G^{(1)}_\pm )$
similarly follows from this square 
$$
\begin{CD}
\ZZ^{\tilde{V}} @>{ L(\tilde{G}_\pm) }>>  \ZZ^{\tilde{V}} \\
@V{  \pi_{(1)} }VV  
    @VV{ \pi_{(1)} }V \\
\ZZ^V @>{  L(G^{(1)}_\pm ) }>> \ZZ^V
\end{CD}
$$
obtained by gluing the two commutative squares in 
\eqref{edge-lift-of-pi-commutative-diagrams} along their common vertical edge.
One wants to show that this inclusion is an equality, but this 
follows from the fact that
$\pi_{(1)}:\ZZ^{\tilde{V}}\mapsto \ZZ^{V}$ is indeed
surjective.

\vskip.1in
\noindent
{\sf Assertion (v).}
This is a completely straightforward verification, left to the reader.

\vskip.1in
\noindent
{\sf Assertion (vi).}
This follows because $\iota$ is a 
signed graph automorphism of $\tilde{G}_\pm$ that involves no sign switches,
only permutations of the coordinates.
\end{proof}

\begin{corollary}
\label{signed-graph-double-cover-splitting} 
Given $\tilde{G}_\pm =\Double(G^{(1)}_\pm, G^{(2)}_\pm)$,
the maps 
$
\ZZ^{\tilde{V}}
\underset{\pi^t_{(i)}}{\overset{\pi_{(i)}}{\rightleftarrows}}
\ZZ^V
$ 
for $i=1,2$ induce morphisms
$$
K(\tilde{G}_\pm) 
\underset{\pi^t_{(i)}}{\overset{\pi_{(i)}}{\rightleftarrows}} 
K(G^{(i)}_\pm) 
$$
with the properties that 
\begin{enumerate}
\item[(a)] $\pi_{(i)}$ is surjective, and
\item[(b)] $\pi_{(i)} \pi^t_{(i)} = 2 \cdot \one_{K(G^{(i)}_\pm)}$.
\end{enumerate}
Hence the $p$-primary component $\Syl_p K(G^{(i)}_\pm)$
splits as a direct summand of $\Syl_p K(\tilde{G}_\pm)$ for odd primes $p$.
\end{corollary}
\begin{proof}
The fact that they induce morphisms follows from the first four assertions of 
Proposition~\ref{technical-proposition}:  one first must
check that they preserve the appropriate sublattices $\im \del_{G_\pm}$, which follow from
(i),(ii), and then that they preserve the further sublattices $\im L(G_\pm)$,
which follow from (iii), (iv).  
The assertion (a) of surjectivity for $\pi_{(i)}$ also 
follows from this, because (i) asserts an equality, not just an inclusion.
To prove assertion (b), one readily checks that one has the same equation  
$\pi_{(i)} \pi^t_{(i)} = 2 \cdot \one_{\ZZ^V}$ already
as operators on $\ZZ^V$.
\end{proof}

Our goal is to be much more precise about 
the kernels of the surjections $\pi_{(i)}$
in Corollary~\ref{signed-graph-double-cover-splitting}.
For this it is convenient to assume that the multigraph
$G$ underlying both $G^{(1)}_\pm, G^{(2)}_\pm$ is connected.  As with 
unsigned graph coverings, this is
a harmless assumption:  whenever $G$ has a nontrivial decomposition
into connected components, there is a corresponding decomposition for the
double cover, and the maps $\pi_{(i)}$ correspondingly
decompose as direct sums.

When $G$ is connected, we will make some
further preparatory assumptions about the connected
component structure of $\tilde{G}_\pm$, beginning with
the following observation.

\begin{proposition}
If $G$ is connected, the underlying multigraph $\tilde{G}$ of
$\tilde{G}_\pm$ can have at most two connected components,
and when there are two components, they are exchanged by
the involutive automorphism $\iota$.
\end{proposition}
\begin{proof}
Fix a base vertex $v$ of $G$, with two lifts $v_+$ or to $v_-$.  
Since every other vertex $u$ of $G$ has a path to $v$ in $G$,
every vertex $u_+, u_-$ in $\tilde{G}$ either has a lifted path
to $v_+$ or to $v_-$ or to both, and hence lies in the component
of one (or both) of $v_+,v_-$.  Note also that $\iota$ must
send the component of $v_+$ to the one of $v_-$.
\end{proof}

This leaves three cases for $\tilde{G}_\pm$ if $G$ is connected:
\begin{enumerate}
\item[{\sf Case 1.}]
The signed graph $\tilde{G}_\pm$ is connected and unbalanced.
\item[{\sf Case 2.}]
The signed graph $\tilde{G}_\pm$ is connected and balanced.
\item[{\sf Case 3.}]
The signed graph $\tilde{G}_\pm$ has two connected components,
exchanged by $\iota$.
\end{enumerate}

In Case 3, we claim one can perform a sequence
of switches at various vertices $v$ of $G^{(2)}_\pm$,
with the effect of exchanging the labels $v_+ \leftrightarrow v_-$
in $\tilde{G}_\pm=\Double(G^{(1)}_\pm,G^{(2)}_\pm)$, until
the two connected components of $\tilde{G}$ have
vertex sets $\{v_+\}_{v \in V}$ and $\{v_-\}_{v \in V}$.
In other words, one can take
$G^{(2)}_\pm = G^{(1)}_\pm$ without loss of generality,
so $\tilde{G}_\pm$ is the {\it disjoint union} 
$G^{(1)}_\pm \sqcup G^{(1)}_\pm$.  
{\it We tacitly make this assumption
whenever in Case 3.}

Although not obvious, we also claim that in Case 2, 
one or the other of $G^{(1)}_\pm$ or $G^{(2)}_\pm$ (but not both)
must be balanced, that is, signed isomorphic to an unsigned
graph; this is proven in Proposition~\ref{Case2-reduction-prop} below.
Hence by swapping their roles, we may assume
that $G^{(1)}_\pm$ is signed isomorphic to an unsigned graph.
One can then perform switches at various vertices $v$ of $G^{(1)}_\pm$,
and accompanying switches at both $v_+,v_-$ in $\tilde{G}_{\pm}$,
so that $G_\pm^{(1)}$ and $\tilde{G}_\pm$ are both unsigned.
{\it We tacitly make this assumption
whenever in Case 2.}

\begin{proposition}
\label{Case2-reduction-prop}
Given $\tilde{G}_\pm =\Double(G^{(1)}_\pm, G^{(2)}_\pm)$, 
create a third signed graph $G^{(1,2)}_\pm$ with same underlying graph $G$ as
$G^{(i)}_\pm$ for $i=1,2$, having voltage assignment
$\beta_{(1,2)}(e) = \beta_{(1)}(e)\beta_{(2)}(e)$ for each $e$ in $E$.

Assuming $G^{(1)}_\pm, G^{(2)}_\pm$ are both unbalanced, then either 
\begin{enumerate}
\item[$\bullet$] $\tilde{G}_\pm$ has two components, if
$G^{(1,2)}_\pm$ is balanced (so we are in a subcase of Case 3), or
\item[$\bullet$] $\tilde{G}_\pm$ is unbalanced,
if $G^{(1,2)}_\pm$ is unbalanced (so we are in Case 1).
\end{enumerate}

In particular, $G^{(1)}_\pm, G^{(2)}_\pm$ both being unbalanced excludes being in Case 2.
\end{proposition}
\begin{proof}
Note that the edges $e$ of $G$ having $\beta_{(1,2)}(e)$ negative are exactly 
the ones whose lifts in $\tilde{G}_\pm$ go across the vertex cut from
$\{v_+\}_{v \in V}$ to $\{v_-\}_{v \in V}$.  Thus whenever
$G^{(1,2)}_\pm$ is balanced, any vertex $v$ of $G$ will have its
two preimages $v_+,v_-$ in $\pi^{-1}(v)$ lying in
different components of $\tilde{G}_\pm$:  
the edges in a path from from $v_+$ to $v_-$ in $\tilde{G}_\pm$ would project to 
an unbalanced cycle for $G_{\pm}^{(1,2)}$.  Hence $\tilde{G}_\pm$ has two
components if $G^{(1,2)}_\pm$ is balanced.


If $G^{(1)}_\pm, G^{(2)}_\pm, G^{(1,2)}_\pm$ are all unbalanced, then we wish to show that
there is a cycle $C$ which is unbalanced for $\tilde{G}_\pm$.
This is the same as showing $C$ is
unbalanced for both $G^{(1)}_\pm$ and $G^{(2)}_\pm$ (and hence balanced
for $G^{(1,2)}_\pm$).  Parity considerations show that 
these are the only possible patterns of balance for cycles $C$ in $G$:

\vskip.1in

\begin{tabular}{|c|c|c|c|} \hline
in $G^{(1)}_\pm$ & in $G^{(2)}_\pm$ & in $G^{(1,2)}_\pm$& in $\tilde{G}_\pm$ \\ \hline\hline
balanced & balanced & balanced & balanced \\ \hline
balanced & unbalanced & unbalanced & (not a cycle) \\ \hline
unbalanced & balanced & unbalanced & (not a cycle) \\ \hline
unbalanced & unbalanced & balanced & unbalanced \\ \hline
\end{tabular}

\vskip.1in
\noindent
Note that, given two cycles $C_1,C_2$ in $G$, since $G$ is connected, 
one can create a third cycle $C_3$ going around $C_1$, following a 
path $P$ to $C_2$, then around $C_2$, and back along the reverse of $P$.  
This $C_3$ will have balance pattern the ``mod 2 sum'' of that for $C_1$ and
$C_2$, reading \{ ``balanced'' , ``unbalanced'' \} as $\{ 0,1 \}$ in $\ZZ_2$.

Now one can complete the argument that there exists a cycle $C$ in
$G$ unbalanced for $\tilde{G}_\pm$, that is, a cycle $C$ matching 
the fourth row of the
table.  We know $G^{(1)}_\pm$ contains some unbalanced cycle $C_1$
and $G^{(2)}_\pm$ contains some unbalanced cycle $C_2$.  One must either
have that one of the two cycles $C_1, C_2$ matches the fourth row of the table, 
in which case we are done,
or $C_1, C_2$ can be combined to create a $C_3$ matching the fourth row of the
table, and again we are done.
\end{proof}

We can now prove the last main result, generalizing 
Theorem~\ref{unsigned-double-cover-theorem} to signed graph double covers.
\begin{theorem}
\label{signed-graph-double-cover-complex}
Given $\tilde{G}_\pm =\Double(G^{(1)}_\pm, G^{(2)}_\pm)$ with underlying
multigraph connected, the maps $\pi^t_{(1)}, \pi_{(2)}$ from 
Corollary~\ref{signed-graph-double-cover-splitting}
fit in a short complex
\begin{equation}
\label{signed-graph-short-complex}
0 \rightarrow K(G^{(1)}_\pm)
    \overset{\pi^t_{(1)}}{\longrightarrow} K(\tilde{G}_\pm)
     \overset{\pi_{(2)}}{\longrightarrow} K(G^{(2)}_\pm)
       \rightarrow 0
\end{equation}
which 
\begin{enumerate}
\item[$\bullet$] in Case 3, is split exact,
\item[$\bullet$] in Case 2, is short exact, and
\item[$\bullet$] in Case 1, is exact at the two ends,
but has homology at the middle term equal to $\ZZ_2$.
\end{enumerate}
In particular, in every case, for all odd primes $p$ one has the splitting
\begin{equation}
\label{odd-prime-splitting}
\Syl_p  K(\tilde{G}_\pm) =
\Syl_p  K(G^{(1)}_\pm) \oplus \Syl_p  K(G^{(2)}_\pm).
\end{equation}
\end{theorem}
\begin{proof}
Note that the asserted splitting \eqref{odd-prime-splitting} will 
follow from the splitting in 
Corollary~\ref{signed-graph-double-cover-splitting},
once the assertions about the short complex are verified.

We first deal with the easy Case 3, where our
preparatory reductions allow one to assume that
$G^{(2)}_\pm = G^{(1)}_\pm$ and $\tilde{G}_\pm = G^{(1)}_\pm \sqcup G^{(1)}_\pm$.
Then setting $K:=K(G^{(1)})$, the sequence \eqref{signed-graph-short-complex}
becomes
$$
\begin{array}{rccccl}
0 \rightarrow &K& \overset{\pi^t_{(1)}}{\longrightarrow} 
      &K \oplus K& \overset{\pi_{(2)}}{\longrightarrow} &K \rightarrow 0\\
&x& \longmapsto &(x,x)&  & \\
& &  &(x,y)&\longmapsto &x-y   \\
\end{array}
$$
which is easily seen to be split exact.

In Cases 1,2, the arguments will resemble each other, and proceed according
to the following plan:
\begin{enumerate}
\item[{\sf Step 1.}]  Show $\pi^t_{(1)}$ maps
$K(G^{(1)}_\pm)$ isomorphically onto
$
\ZZ^{\tilde{V}}_{\sym,=0}/L(\tilde{G}_\pm)(\ZZ^{\tilde{V}}_\sym) 
$ 
in Case 2,\\ 
and isomorphically 
onto an index 2 subgroup of
$
\ZZ^{\tilde{V}}_{\sym}/L(\tilde{G}_\pm)(\ZZ^{\tilde{V}}_\sym) 
$ 
in Case 1.\\
\item[{\sf Step 2.}]  Show that 
$$
\ker\left( K(\tilde{G}_\pm) 
            \overset{\pi_{(2)}}{\rightarrow} 
              K(G^{(2)}_\pm) \right)
=
\begin{cases}
\ZZ^{\tilde{V}}_{\sym,=0}/L(\tilde{G}_\pm)(\ZZ^{\tilde{V}}_\sym) 
 & \text{ in Case 2},\\
\ZZ^{\tilde{V}}_{\sym}/L(\tilde{G}_\pm)(\ZZ^{\tilde{V}}_\sym) 
 & \text{ in Case 1}.\\
\end{cases}
$$
\end{enumerate}
Note that these would imply the
assertions of Case 1 and Case 2 from the theorem.
\vskip.2in
\noindent
{\sf Step 1.}
In Case 2, starting with 
the commuting square of Proposition~\ref{technical-proposition}(iv),
$$
\begin{CD}
\ZZ^{\tilde{V}} @>{ L(\tilde{G}_\pm)}>>  \ZZ^{\tilde{V}} \\
@A{  \pi^t_{(i)} }AA  
    @AA{ \pi^t_{(i)} }A \\
\ZZ^V @>{  L(G^{(i)}_\pm) }>> \ZZ^V
\end{CD}
$$
note that its bottom horizontal map restricts to a map 
$
\ZZ^V \overset{L(G^{(1)}}{\longrightarrow} \im \partial_{G^{(1)}_\pm},
$
since $\im L(G^{(1)}) \subset \im \partial_{G^{(1)}_\pm}$.  
Note also that its left vertical map 
restricts to an isomorphism 
$\ZZ^V \overset{\pi^t_{(1)}}{\longrightarrow} \ZZ^{\tilde{V}}_\sym$ 
according to Proposition~\ref{technical-proposition}(v).
Since we are in Case 2, 
Proposition~\ref{connected-signed-graph-im-of-del} implies
$\im \partial_{G^{(1)}_\pm}=\ZZ^V_{=0}$,
and one can easily
check that this isomorphism 
$\ZZ^V \overset{\pi_{(1)}^t}{\rightarrow} \ZZ^{\tilde{V}}_\sym$
restricts to an isomorphism sending $\im \partial_{G^{(1)}_\pm} (=\ZZ^V_{=0})$ 
isomorphically onto $\ZZ^{\tilde{V}}_{\sym,=0}$.
Thus one deduces that the commuting square of 
Proposition~\ref{technical-proposition}(iv) restricts to the following
square in which both vertical maps are isomorphisms induced by $\pi_{(1)}^t$:
$$
\begin{CD}
\ZZ^{\tilde V}_{\sym} @>{L(\tilde{G}_\pm)}>>  \ZZ^{\tilde{V}}_{\sym,=0} \\
@A{\pi_{(1)}^t}AA
@AA{\pi_{(1)}^t}A \\
\ZZ^V @>{L(G^{(1)}_\pm) }>> \im\partial_{G^{(1)}_\pm}  
\end{CD}
$$
\noindent
The five-lemma shows
$\pi^t_{(1)}$ induces an isomorphism
from the cokernel $K(G^{(1)}_\pm)$ 
of the bottom horizontal row here to the cokernel
$\ZZ^{\tilde{V}}_{\sym,=0}/L(\tilde{G}_\pm)(\ZZ^{\tilde V}_{\sym})$
of the top horizontal row here, as desired in Step 1 for Case 2.

When we are in Case 1, 
Proposition~\ref{connected-signed-graph-im-of-del} implies
that $\im \partial_{G^{(1)}_\pm} =\ZZ^V_{\equiv 0 \bmod{2}}$ is
an index $2$ sublattice of $\ZZ^V$, and hence 
$K(G^{(1)}_\pm)= \im \partial_{G^{(1)}_\pm}/\im L(G^{(1)}_\pm)$ is
an index $2$ subgroup of $\ZZ^V/\im L(G^{(1)}_\pm)$.  Therefore our
stated goal for Step 1 in Case 1 would be achieved if one could
show that $\pi^t_{(1)}$ maps $\ZZ^V/\im L(G^{(1)}_\pm)$ isomorphically
onto 
$
\ZZ^{\tilde{V}}_{\sym}/L(\tilde{G}_\pm)(\ZZ^{\tilde{V}}_\sym).
$ 
This is argued similarly to Case 2:  restricting
the commuting square of Proposition~\ref{technical-proposition}(iv)
gives this square with vertical isomorphisms
$$
\begin{CD}
\ZZ^{\tilde V}_{\sym} @>{L(\tilde{G}_\pm)}>>  \ZZ^{\tilde{V}}_{\sym} \\
@A{\pi_{(1)}^t}AA
@AA{\pi_{(1)}^t}A \\
\ZZ^V @>{L(G^{(1)}_\pm) }>> \ZZ^V  
\end{CD}
$$
and then the five-lemma shows 
$\pi^t_{(1)}$ induces an isomorphism from
the cokernel $\ZZ^V/\im L(G^{(1)}_\pm)$
of the bottom horizontal row to the cokernel
$\ZZ^{\tilde{V}}_{\sym}/L(\tilde{G}_\pm)(\ZZ^{\tilde V}_{\sym})$,
of the top horizontal row, as desired.

\vskip.1in
\noindent
{\sf Step 2.}
In both Cases 1,2, start reformulating 
$
\ker\left( K(\tilde{G}_\pm) 
            \overset{\pi_{(2)}}{\rightarrow} 
              K(G^{(2)}_\pm) \right)
$
via a diagram of short complexes  
\begin{equation}
\label{nine-diagram}
\begin{array}{rccccccl}
 & 0 & & 0 & & 0 & \\
 & \downarrow & & \downarrow & & \downarrow & \\
0 \rightarrow 
&\im L(\tilde{G}_\pm) \cap \ZZ^{\tilde{V}}_\sym&
    \longrightarrow & \im L(\tilde{G}_\pm) & 
      \overset{\pi_{(2)}}{\longrightarrow} &\im L(G^{(2)}_\pm) & 
        \rightarrow 0 \\
 & \downarrow & & \downarrow & & \downarrow & \\
0 \rightarrow 
&\im \del_{\tilde{G}_\pm} \cap \ZZ^{\tilde{V}}_\sym&
    \longrightarrow & \im \del_{\tilde{G}_\pm} & 
      \overset{\pi_{(2)}}{\longrightarrow} &\im \del_{G^{(2)}_\pm} & 
        \rightarrow 0 \\
 & \downarrow & & \downarrow & & \downarrow & \\
0 \rightarrow 
&\im \del_{\tilde{G}_\pm} \cap \ZZ^{\tilde{V}}_\sym \left/
    \im L(\tilde{G}_\pm) \cap \ZZ^{\tilde{V}}_\sym  \right.&
    \longrightarrow & K(\tilde{G}_\pm) & 
      \overset{\pi_{(2)}}{\longrightarrow} & K(G^{(2)}_\pm) & 
        \rightarrow 0 \\
 & \downarrow & & \downarrow & & \downarrow & \\
 & 0 & & 0 & & 0. & 
\end{array}
\end{equation}
in which the vertical sequences are all exact by definition.
We argue here why its horizontal rows 1,2,3 are 
also exact.  
The horizontal maps in rows 1 and 2 
come from an exact sequence derived from 
Proposition~\ref{technical-proposition}(v) 
\begin{equation}
\label{ambient-exact-sequence}
0 \longrightarrow \ZZ^{\tilde V}_\sym \longrightarrow \ZZ^{\tilde{V}}
    \overset{\pi_{(2)}}{\longrightarrow} \ZZ^V 
      \longrightarrow 0 
\end{equation}
which one intersects with the first two terms 
in this tower of inclusions:
$
\im  L(\tilde{G}_\pm)
\subset \im \del_{\tilde{G}_\pm} 
\subset \ZZ^{\tilde{V}}.
$
Thus rows 1 and 2 are exact at their left and 
middle positions due to the exactness of \eqref{ambient-exact-sequence}.
They are exact at their right positions due to 
Proposition~\ref{technical-proposition}(i) and (iii).
Hence rows 1 and 2 are exact, and then
by the nine-lemma, Row 3 
is also exact.

Exactness of Row 3 lets one reformulate
\begin{equation}
\label{nine-lemma-consequence}
\ker\left( K(\tilde{G}_\pm) 
            \overset{\pi_{(2)}}{\rightarrow} 
              K(G^{(2)}_\pm) \right) 
= \im \del_{\tilde{G}_\pm} \cap \ZZ^{\tilde{V}}_\sym  \left/
    \im  L(\tilde{G}_\pm) \cap \ZZ^{\tilde{V}}_\sym  \right. .
\end{equation}
Next we further simplify the numerator and denominator 
on the right side of \eqref{nine-lemma-consequence}.
First we claim that one can reformulate the numerator 
on the right side of \eqref{nine-lemma-consequence} as
$$
\im \del_{\tilde{G}_\pm} \cap \ZZ^{\tilde{V}}_\sym 
=
\begin{cases}
\ZZ^{\tilde{V}}_\sym & \text{ in Case 1},\\
\ZZ^{\tilde{V}}_{=0} \cap \ZZ^{\tilde{V}}_{\sym} 
=:\ZZ^{\tilde{V}}_{\sym,=0} & \text{ in Case 2}.
\end{cases}
$$
due to Proposition~\ref{connected-signed-graph-im-of-del}.
In Case 1, so that $\tilde{G}_\pm$ is unbalanced,
this proposition asserts 
that $\im \del_{\tilde{G}_\pm} =\ZZ^{\tilde{V}}_{\equiv 0 \bmod{2}}$,
which contains $\ZZ^{\tilde{V}}_\sym$ as a sublattice.
In Case 2, so that our preparatory reductions have us
assume $G^{(1)}_\pm$ and $\tilde{G}_\pm$ are unsigned,
this proposition asserts that 
$\im \del_{\tilde{G}_\pm} =\ZZ^{\tilde{V}}_{=0}$.

We next argue that one can reformulate the denominator
on the right side of \eqref{nine-lemma-consequence} as
$$
\im L(\tilde{G}_\pm) \cap \ZZ^{\tilde{V}}_\sym 
=
L(\tilde{G}_\pm)\left( \ZZ^{\tilde{V}}_\sym \right).
$$
To see this, note that its elements are those of the form
$L(\tilde{G}_\pm)(x)$ lying in
$\ZZ^{\tilde{V}}_\sym$, meaning that
$$
0 = \iota L(\tilde{G}_\pm)(x)
    -  L(\tilde{G}_\pm)(x)   
=L(\tilde{G}_\pm)( \iota(x) -x )
$$
via Proposition~\ref{technical-proposition}(vi).  
In Case 1, so that $\tilde{G}_\pm$ is unbalanced, 
equation \eqref{signed-Laplacian-nonsingular} in the proof 
Proposition~\ref{matrix-tree-like-prop} showed
that $L(\tilde{G}_\pm)$ is invertible, so
this is equivalent to $0=\iota(x)-x$, that is, $x$ lies in
$\ZZ^{\tilde{V}}_\sym$, as claimed.
In Case 2, because
$\tilde{G}_\pm$ is unsigned, one has $\delT^t_{\tilde{G}_\pm}=
\del^t_{\tilde{G}_\pm}$.
Thus since $\iota(x) -x$ lies in the kernel of
$$
L(\tilde{G}_\pm)=
\del_{\tilde{G}_\pm} \delT^t_{\tilde{G}_\pm}
=
\del_{\tilde{G}_\pm} \del^t_{\tilde{G}_\pm},
$$
it must also lie in the kernel of 
$\del^t_{\tilde{G}_\pm}$.  As $\tilde{G}$ is connected,
this means $\iota(x)-x$ has all its coordinates equal, which then
forces $\iota(x)-x=0$, that is, again $x$ lies in $\ZZ^{\tilde{V}}_\sym$.
Hence in this case one has
$$
\im  L(\tilde{G}_\pm) \cap \ZZ^{\tilde{V}}_\sym  =
L(\tilde{G}_\pm) ( \ZZ^{\tilde{V}}_\sym ).
$$
%

Thus we have reformulated the kernel
\eqref{nine-lemma-consequence} as
\begin{equation}
\label{reformulated-nine-lemma-consequence}
\ker\left( K(\tilde{G}_\pm) 
            \overset{\pi_{(2)}}{\rightarrow} 
              K(G^{(2)}_\pm) \right) =
\begin{cases}
\ZZ^{\tilde{V}}_\sym 
/L(\tilde{G}_\pm) \left( \ZZ^{\tilde{V}}_\sym \right)
& \text{ in Case 1},\\
\ZZ^{\tilde{V}}_{\sym,=0} 
/L(\tilde{G}_\pm) \left( \ZZ^{\tilde{V}}_{\sym} \right)
& \text{ in Case 2}
\end{cases}
\end{equation}
which was exactly our goal in Step 2.
\end{proof}

\section{Two applications of signed graph double covers}
\label{signed-graph-double-cover-application-section}
We conclude with two applications of 
Theorem~\ref{signed-graph-double-cover-complex},

\subsection{Application:
Crowns revisited}
\label{revisited-crown-section}

Recall from Example~\ref{crown-example}
that the unsigned multigraph $\crown_n^{(k)}$ is obtained from a complete 
bipartite graph $K_{n,n}$ by removing a perfect matching $M$ of edges,
and then replacing $M$ with $k$ copies of this same matching,
so that $\crown_n^{(k)}$ is $K_{n,n}$ together with $k-1$ added extra copies
of each edge in the matching $M$.  Also recall that
Corollary~\ref{first-crown-corollary} proved the following formula
$$
K(\crown_n^{(k)}) 
\cong \ZZ_n^{n-2} \oplus \ZZ_{n-2+2k}^{n-2} \oplus \ZZ_{(n-1+k)(n-2+2k)}^{n-2},
$$
under assumptions that 
\begin{itemize} 
\item $n$ is odd, and 
\item $k$ is even.
\end{itemize}

\begin{corollary}
Assuming $n$ is odd, this formula
for $K(\crown_n^{(k)})$ is correct, regardless of the parity of $k$.
\end{corollary}
\begin{proof}
Now that we can allow half-loops in our graphs, 
regardless of the parity of $k$, one can define the
unsigned multigraph 
$K_n^{(\frac{k}{2})}$ 
to be obtained from a complete graph
$K_n$ by adding $k$ copies of a (positive) half-loop to each vertex $v$.
Consider this as a signed graph $G^{(1)}_\pm:=K_n^{(\frac{k}{2})}$ 
and introduce its negative $G^{(2)}_\pm:=-K_n^{(\frac{k}{2})}$ as the signed graph 
obtained from a complete graph having all negative edges
by adding $k$ copies of a negative half-loop to each vertex $v$.

One can then check that $\crown_n^{(k)}$ is exactly the 
associated signed graph double covering
$
\Double(G^{(1)}_\pm,G^{(2)}_\pm).
$
Thus Case 2 of Theorem~\ref{signed-graph-double-cover-complex}
recovers a short exact sequence generalizing \eqref{crown-short-exact-sequence}
\begin{equation}
\label{signed-crown-short-exact-sequence}
0 \rightarrow K(K_n^{(\frac{k}{2})})
) \rightarrow  
    K(\crown_n^{(k)}) \rightarrow 
      K(-K_n^{(\frac{k}{2})})  \rightarrow 0. 
\end{equation}
The remainder of the proof of Corollary~\ref{first-crown-corollary}
showing that
$$
\begin{aligned}
K(K_n^{(\frac{k}{2})})
&=\ZZ_n^{n-2}\\ 
K(-K_n^{(\frac{k}{2})}) &= \ZZ_{n-2+2k}^{n-2} \oplus \ZZ_{(n-1+k)(n-2+2k)}
\text{ for }n\text{ odd},
\end{aligned}
$$
and that the sequence splits for $n$ odd,
still applies unchanged.
\end{proof}

\begin{remark}
\label{Machacek-n-even-crown-remark}
When {\it $n$ is even}, things are trickier.
However, with a bit more work the second author 
was able to use these methods to
derive the following 
formula for the case when 
{\it $n$ is even and $\gcd(k-1,n)=1$}:
$$
K(\crown_n^{(k)}) 
=
 \ZZ_{n - 2 + 2k}
\oplus \ZZ_{n(n - 2 + 2k)}^{n - 3} 
\oplus \ZZ_{n(n - 1 + k)(n - 2 + 2k)}.
$$
See Tseng \cite[Proposition 8.4]{Tseng}.  
In particular, when $k=0$ and $k=2$, this result applies to {\it all} even $n$,
and the $k=0$ case recovers the rest of the answer 
\eqref{Machacek's-crown-calculation} computed by 
Machacek \cite[Theorem 14]{Machacek}, that was
discussed for $n$ odd already in Example~\ref{crown-example} above.
\end{remark}

\subsection{Application:
Reinterpreting Bai's calculation for the $n$-cube}
\label{Bai-section}

\begin{definition}
Let $Q_n$ denote the unsigned graph of the {\it $n$-dimensional cube},
that is, its vertices are all binary vectors in $\{0,1\}^n$, and
two such vertices lie on an edge if they differ in exactly one coordinate.
\end{definition}

H. Bai calculated the structure of the $p$-primary component $\Syl_p K(Q_n)$
for all odd primes $p$, using an induction on $n$, 
that proceeded via consideration of cokernels
for a larger family of matrices.  
We use Theorem~\ref{signed-graph-double-cover-complex}
to reinterpret his calculation here geometrically, identifying 
these matrices as Laplacians for a larger family of signed graphs,
involved in a family of double covers.

\begin{definition}
For nonnegative integers $m,n$, consider the
the signed graph $Q_n^{(m)}$ whose underlying unsigned graph is the $n$-cube
$Q_n$ with $m$ added half-loops at each vertex, and with
voltage assignment $\beta$ in which 
all (nonloop) cube edges $e$ of $Q_n$ have $\beta(e)=+$,
and all the half-loops $e$ have $\beta(e)=-$.
\end{definition}

\begin{proposition}
\label{Bai-reinterpreted-proposition}
For $n \geq 1$, and for each odd prime $p$ one has
$$
\Syl_p K( Q_n^{(m)} ) 
\quad \cong \quad
\Syl_p K( Q_{n-1}^{(m)} )
\quad \oplus \quad
\Syl_p K( Q_{n-1}^{(m+1)} ).
$$
\end{proposition}
\begin{proof}
Consider the signed graph double covering
$\tilde{G}_{\pm}=\Double( G^{(1)}_\pm, G^{(2)}_\pm )$
in which $G^{(1)}_\pm$ is obtained from $Q_{n-1}^{(m)}$
by adding one positive half-loop at each vertex, and
where $G^{(2)}_\pm=Q_{n-1}^{(m+1)}$.  Then
$\tilde{G}_\pm=Q_n^{(m)}$, since each positive half-loop on the
vertices of $G^{(1)}_\pm$ will be double covered by
a positive edge of $\tilde{G}$ ``stretched out'' into the
$n^{th}$ coordinate direction.
The example with
$n=3, m=2$ is pictured in Figure~\ref{cube-squashing-double-cover-figure},
with $\tilde{G}_\pm=Q^{(2)}_3 \rightarrow G^{(1)}_\pm$,
and $G^{(2)}_\pm=Q^{(3)}_2$.

As positive half-loops give rise to zero columns of
$\del$ and $\delT$, they have no effect on $K(G_\pm)$,
and hence $K(G^{(1)}_\pm) =K( Q_{n-1}^{(m)} )$.  Therefore
the splitting of \eqref{odd-prime-splitting} implies
the assertion of the proposition.
\end{proof}

\begin{corollary}(cf. Bai \cite[\S 2]{Bai})
For $n \geq 0$, and for each odd prime $p$ one has
$$
\Syl_p K( Q_n^{(m)} ) =
\Syl_p \left( \bigoplus_{k=0}^n \ZZ_{k+m}^{\binom{n}{k}} \right),
$$
and in particular, when $m=0$, the $n$-cube $Q_n$ has
$$
\Syl_p K( Q_n ) =
\Syl_p \left( \bigoplus_{k=0}^n \ZZ_k^{\binom{n}{k}} \right).
$$
\end{corollary}
\begin{proof}
Induct on $n$.  If $n=0$, 
the signed graph $Q_0^{(m)}$ has one vertex with $m$ negative half-loops,
so that
$$
\begin{aligned}
\del&=
\left[
\begin{matrix}
2 &2& \cdots&2
\end{matrix}
\right] \\
\delT&=
\left[
\begin{matrix}
1 &1& \cdots&1
\end{matrix}
\right] \\
K(Q_0^{(m)})& = \im \del/ \im \del \delT^t
= 2\ZZ / 2m\ZZ \cong \ZZ_m 
\left( = \bigoplus_{k=0}^n \ZZ_{k+m}^{\binom{n}{k}} \text{ for }n=0\right).
\end{aligned}
$$

In the inductive step, where $n>1$, applying
Proposition~\ref{Bai-reinterpreted-proposition}
gives
$$
\begin{aligned}
\Syl_p K( Q_n^{(m)} ) \quad
& =\Syl_p K( Q_{n-1}^{(m)} ) \quad \oplus \quad \Syl_p K( Q_{n-1}^{(m+1)} ) \\
& = \Syl_p \left( \bigoplus_{k=0}^{n-1} \ZZ_{k+m}^{\binom{n-1}{k}} \right) 
\oplus
\Syl_p \left( \bigoplus_{k=0}^{n-1} \ZZ_{k+m+1}^{\binom{n-1}{k}} \right)\\
&= \Syl_p \left( \ZZ_m^{\binom{n}{0}} \quad \oplus \quad
\bigoplus_{k=1}^{n-1} \ZZ_{k+m}^{\binom{n-1}{k}+\binom{n-1}{k-1}} \quad \oplus\quad 
\ZZ_{m+n}^{\binom{n-1}{n-1}} \right) \\
& =\Syl_p \left( \bigoplus_{k=0}^n \ZZ_{k+m}^{\binom{n}{k}} \right)
\end{aligned}
$$
where the second equality used the inductive hypothesis.
\end{proof}


\begin{thebibliography}{99}

\bibitem{Abdollahi}
A. Abdollahi,
Determinants of adjacency matrices of graphs.
{\it Transactions on Combinatorics} {\bf 1} (2012), 9--16.

\bibitem{BHN}
R. Bacher, P. de la Harpe, and T. Nagnibeda, 
The lattice of integral cuts and lattice of integral flows of a finite graph. 
{\it Bull. Soc. Math. France} {\bf 125} (1997), 167--198.

\bibitem{Bai}
H. Bai,
On the critical group of the {$n$}-cube,
{\it Linear Algebra Appl.} {\bf 369} (2003), 251--261.

\bibitem{BakerNorine1}
M. Baker and S. Norine, 
Riemann-Roch and Abel-Jacobi theory on a finite graph,
{\it Adv. Math.} {\bf 215} (2007), 766 -- 788. 

\bibitem{BakerNorine2}
M. Baker and S. Norine,
Harmonic morphisms and hyperelliptic graphs.
{\it Int. Math. Res. Not. IMRN} (2009), 2914 -- 2955. 


\bibitem{Berget}
A. Berget,
Critical groups of graphs with reflective symmetry,
{\tt arXiv:1208.0632}


\bibitem{BMMPR}
A. Berget, A. Manion, M. Maxwell, A. Potechin, and V. Reiner,
The critical group of a line graph,
to appear in {\it Annals Combin.}; {\tt arXiv:0904.1246}


\bibitem{Berman}
K.A. Berman,
Bicycles and spanning trees,
{\it SIAM J. Algebraic Discrete Methods} {\bf 7} (1986), 1--12. 

\bibitem{Biggs}
N. Biggs,
Algebraic potential theory on graphs,
{\it Bull. London Math. Soc.} {\bf 29} (1997), 641 -- 682. 


\bibitem{Jacobson}
B. Jacobson, A. Niedermaier, and V. Reiner, 
Critical groups for complete multipartite graphs and Cartesian products of complete graphs,
{\it J. Graph Theory} {\bf 44} (2003), 231--250.

\bibitem{GodsilRoyle}
C. Godsil and G. Royle, 
Algebraic graph theory. 
{\it Graduate Texts in Mathematics} {\bf 207}. Springer-Verlag, New York, 2001

\bibitem{Gross}
J. L. Gross, and T.W. Tucker, 
Generating all graph coverings by permutation voltage assignments,
{\it Discrete Math.} {\bf 18} (1977), 273--283.

\bibitem{Lorenzini}
D. Lorenzini,
A finite group attached to the Laplacian of a graph,
{\it Discrete Math.} {\bf 91} (1991), 277--282.

\bibitem{Machacek}
J. Machacek, 
The critical group of a line graph: the bipartite case,
Bachelors thesis, U. Minnesota, 2011; \\
see {\tt www.math.umn.edu/\textasciitilde{}reiner/HonorsTheses/Machacek\_thesis.pdf}.


\bibitem{Perkinson}
D. Perkinson, J. Perlman, and J. Wilmes,
Primer for the algebraic geometry of sandpiles,
{\tt arXiv:1112.6163}.
 
\bibitem{Treumann}
D. Treumann,
Functoriality of Critical Groups,
Bachelors thesis, U. Minnesota, 2002; \\
see {\tt www.math.umn.edu/\textasciitilde{}reiner/HonorsTheses/Treumann\_thesis.pdf}.

\bibitem{Tseng}
D. Tseng,
Graph coverings and critical groups of signed graphs and voltage graphs,
Univ. of Minnesota REU report, July 2012;\\
see {\tt www.math.umn.edu/\textasciitilde{}reiner/REU/Tseng2012.pdf}.

\bibitem{Urakawa}
H. Urakawa, 
A discrete analogue of the harmonic morphism and Green kernel comparison theorems, 
{\it Glasg.  Math. J.} {\bf 42} (2000), 319 -- 334.

\bibitem{Waller}
D.A. Waller, 
Double covers of graphs
{\it Bull. Austral. Math. Soc.} {\bf 14} (1976), 233 -- 248

\bibitem{Zaslavsky}
T. Zaslavsky,
Signed graphs,
{\it Discrete Appl. Math.} {\bf 4} (1982), 47--74;
erratum in {\it Discrete Appl. Math.} {\bf 5} (1983), 248.


\end{thebibliography}

\end{document}